\documentclass[11pt]{article}

\usepackage[margin=1in]{geometry}
\usepackage{soul}

\usepackage[utf8]{inputenc}
\usepackage{dsfont}

\usepackage{amsmath}
\usepackage{amsfonts}
\usepackage{amssymb}
\usepackage{amsthm}
\usepackage{xcolor}
\usepackage{ulem}
\usepackage{graphicx}

\DeclareMathOperator{\Sq}{\mathrm{Sq}}

\newtheorem{thm}{Theorem}

\newtheorem{rem}{Remark}

\newtheorem{lem}{Lemma}
\newtheorem{defn}{Definition}

\newtheorem{fact}{Fact}

\newtheorem{question}{Question}

\newtheorem*{thmalt}{Theorem \ref{refin}'}

\DeclareMathOperator{\Pow}{\mathrm{Pow}}

\DeclareMathOperator{\x}{\mathbf{x}}

\newcommand{\norm}[1]{\left\lvert\left\lvert #1 \right\rvert\right\rvert}
\newcommand{\abs}[1]{\left\lvert #1 \right\rvert}

\title{Equidistribution of orbits at polynomial times in rigid dynamical systems}

\author{Kosma Kasprzak}
\date{\empty}

\begin{document}

\maketitle
\begin{abstract}
We study distribution of orbits sampled at polynomial times for uniquely ergodic topological dynamical systems $(X, T)$. 
First, we prove that if there exists an increasing sequence $(q_n)$ for which the rigidity condition
\[
\max_{t<q_{n+1}^{4/5}}\sup_{x\in X}d(x, T^{tq_n}x)=o(1)
\]
is satisfied, then all square orbits $(T^{n^2}x)$ are equidistributed (with respect to the only invariant measure). We show that this rigidity condition might hold for weakly mixing systems, and so as a consequence we obtain first examples of weakly mixing systems where such an equidistribution holds. We also show that for integers $C>1$ a much weaker rigidity condition 
\[
\max_{t<q_n^{C-1}}\sup\limits_{x\in X}d\left(x, T^{tq_n}x\right)=o(1)
\]
implies density of all orbits $(T^{n^C}x)$ in totally uniquely ergodic systems, as long as the sequence $(\omega(q_n))$ is bounded.

\end{abstract}

\section{Introduction}

\subsection{Ergodic averages along sequences}\label{intr}

In this paper we consider topological dynamical systems, that is, pairs $(X, T)$ where $X$ is a compact metric space and $T:X\to X$ a homeomorphism. We study the averages
\begin{equation}\label{ergavg2}
\frac{1}{N}\sum\limits_{n<N}f(T^{a_n}x)
\end{equation}
for functions $f\in C(X)$ and sequences $(a_n)$ of natural numbers, specifically sequences of the form $a_n=P(n)$ for polynomials $P$. We are particularly interested in convergence of such averages at every point $x\in X$. Recall that convergence of these averages in the classical case, i.e. when $a_n=n$, follows from unique ergodicity of $T$.

One can also study the case when $(X, T)$ comes equipped with an invariant measure, and ask about convergence of (\ref{ergavg2}) at almost every point (with respect to this measure). For the classical ergodic averages convergence of \eqref{ergavg2} follows from the celebrated Birkhoff ergodic theorem.
This theorem has been extended by Bourgain and Wierdl to averages along sequences $(P(n))$ for polynomials $P$, and along the sequence $(p_n)$ of prime numbers. The authors use fine harmonic analysis, in particular oscillation inequalities, to show that for the mentioned sequences averages $(\ref{ergavg2})$ converge at almost every $x\in X$ as long as $f\in L^p(X)$ for $p>1$. However, in general the limit is not determined.

Notice, that in uniquely ergodic systems convergence of (\ref{ergavg2}) everywhere to the integral of $f$ (with respect to the unique measure) is simply the statement, that each $(a_n)$-orbit $(T^{a_n}x)$ equidistributes in $X$ (with respect to the unique invariant measure).
Convergence for every point $x\in X$ is a much more delicate problem than almost everywhere convergence. Indeed, for $(a_n)$ strictly increasing and with density zero, Pavlov provided in (\cite{pavlov}, Theorems 1.2 and 1.3) general symbolic examples of well-behaved systems where the averages (\ref{ergavg2}) diverge for some $x\in X$. Another general class of counterexamples, in the form of skew product systems over circle rotations, was given by Kanigowski, Lemańczyk and Radziwiłł in (\cite{PNT}, Section 11). In particular, these results encompass primes and polynomials, and it seems that to establish convergence everywhere one needs quantitative information about the system, rather than just ergodic properties.

There are other sequences, e.g. $a_n=\Omega(n)$, where qualitative information is enough to ensure equidistribution of $(T^{\Omega(n)}x)$, as shown by Bergelson and Richter in \cite{bergelson}. Another instance where qualitative information might be sufficient is the problem of M{\"o}bius orthogonality (Sarnak's conjecture), where it is often enough to study joinings of different powers of the system, see e.g. the survey \cite{sarnak} by Ferenczi, Kułaga-Przymus and Lemańczyk.

For primes, everywhere convergence has been obtained in several classes using quantitative information; see the Introduction to \cite{PNT} and referencecs therein.

Let us also mention a recent result by Forni, Kanigowski and Radziwiłł, who, using quantitative bounds on so called type II sums, show in \cite{FKR} that (\ref{ergavg2}) converge along semiprimes for all $x\in X$ for the time one map of the horocycle flow in certain quotients of the hyperbolic plane. In fact, they show that each semiprime orbit is either periodic or equidistributed in $X$.

In the case of polynomial sequences much less is known. 
A general method developed by Venkatesh in \cite{venkatesh} allows results for $(a_n)$ of the type $n^{1+\delta}$ with small $\delta>0$, using quantitative ergodicity and quantitative mixing of the system (see also \cite{FFT} by Flaminio, Forni and Tanis). These methods are however not applicable for $\delta>1/2$, so in particular do not give results for polynomials.

The first everywhere convergence result for $a_n=P(n)$ for polynomials $P$ is convergence of (\ref{ergavg2}) for circle rotations, obtained by Weyl. In \cite{leibman} Leibman generalized this to all nilsystems, and later Green and Tao obtained quantitative bounds in this direction in \cite{green-tao}. The methods used in these results rely essentially on variants of the van der Corput inequality, using the existence of generalized eigenfunctions to reduce the degree of the considered polynomial to the linear case. 

In particular, this approach cannot be used in weakly mixing systems, where there are no generalized eigenfunctions.  Because of this, there are no known weakly mixing systems where polynomial orbits equidistribute. Additionally, there are no known systems where equidistribution holds for all $(n^2)$-orbits, but not all $(n^3)$-orbits.

The goal of this paper is to introduce new techniques of studying equidistribution of orbits at polynomials times, applicable in particular to weakly mixing systems. For this purpose, we impose quantitative rigidity conditions on the system $(X, T)$, and use them to reduce the question of equidistribution to a number-theoretic problem concerning the distribution of $(P(n))$ in residue classes modulo integers. Specifically, we need to examine how many terms of this sequence, when reduced modulo a given integer, fall into a given arithmetic progression.
For this purpose, we need bounds for exponential sums in short intervals with polynomial phases.

\subsection{Main results}
Our main result is 
\begin{thm}\label{strrigid}
Assume that $(X, T)$ is a uniquely ergodic dynamical system and $(q_n)$ is an increasing sequence of positive integers satisfying $\gcd(q_n, q_{n+1})=1$ such that
\begin{equation}\label{rigid}
\max_{t<q_{n+1}^{4/5}}\sup_{x\in X}d(x, T^{tq_n}x)=o(1).
\end{equation}
Then for any function $f\in C(X)$ we have
\[
\lim_{N\to\infty}\frac{1}{N}\sum\limits_{i<N}f(T^{i^2}x)=\int_X f d\mu
\]
uniformly in $x\in X$, where $\mu$ is the unique ergodic measure on $X$.
\end{thm}
In Section \ref{examples} we produce examples of skew product systems and special flows (both over irrational rotations) satisfying these assumptions. In particular, we give examples of weakly mixing systems where all square orbits equidistribute.

Interestingly, our method does not seem to generalize to sequences $(P(n))$ for polynomials $P$ of degree higher than 3. This is because the condition (\ref{rigid}) is quite strong, since we require $T^{tq_n}$ to be close to the identity for $t$ of size depending on $q_{n+1}$. While we could produce analogous results to Theorem \ref{strrigid} for sequences $(P(n))$, the required rigidity condition would already be too strong to be satisfied by any nontrivial system. We discuss this further in Section \ref{other}. 

For polynomials of degree 3, our methods barely apply under much more restrictive conditions on the dynamical system, but it seems they are still compatible with weakly mixing examples in the form of special flows. We plan to examine this case in future work.\\

We are also interested in more natural weaker rigidity conditions, where we require $T^{tq_n}$ to be close to identity only for $t$ depending on $q_n$ rather than $q_{n+1}$. In Section \ref{weakrig} we prove the following result:
\begin{thm}\label{refin}
    Let $(X, T)$ be a totally uniquely ergodic measure preserving system. Let $C\geqslant 2$ be a natural number. Assume, that for some $k\in\mathbb{N}$ and an increasing sequence $(q_n)$ of natural numbers the following conditions are satisfied:
\begin{enumerate}
    \item for any $n$ we have $\omega(q_n)\leqslant k$,
    \item we have
    \begin{equation}\label{rigweak}
    \max_{t<q_n^{C-1}}\sup\limits_{x\in X}d\left(x, T^{tq_n}x\right)=o(1).
    \end{equation}
    \end{enumerate}
    Then for any $x\in X$ and any function $f\in C(X)$ we have
    \[
    \lim_{n\to\infty}\frac{1}{q_n}\sum\limits_{i< q_n}f(T^{i^C}x)=\int\limits_{X}f \mathrm{d}\mu
    \]
    uniformly in $x\in X$.
\end{thm}
Notice, that this in particular implies density of the orbit $(T^{n^C}x)$
for any starting point $x\in X$.

Again, the method does not seem to work for general polynomials in place of $n^C$. This is because we rely on the fact, that taking $C$-th powers is a homomorphism of the multiplicative groups $(\mathbb{Z}/p\mathbb{Z})^{*}$.

We also outline the limitations of this method, by providing certain counterexamples adapted from (\cite{PNT}, Section 11). We show that, under the weaker type of rigidity conditions we consider in Section \ref{weakrig}, one cannot hope for convergence of the whole sequence of ergodic averages along squares, rather than just the subsequence of averages up to $q_n$. We are also interested in the necessity of the condition that $(\omega(q_n))$ is bounded -- although we cannot prove that it is necessary, we show that this is the case for a stronger version of Theorem \ref{refin}, which we use in the proof of Theorem \ref{refin}.

\subsection{Open questions}

We naturally wonder whether the limitations of our methods can be circumvented:
\begin{question}
    Does a result analogous to Theorem \ref{strrigid} (possibly with a stronger rigidity condition) hold for polynomials of higher degree than 3?
\end{question}

\begin{question}
    Does a result analogous to Theorem \ref{refin} hold for other polynomials than $P(n)=n^C$?
\end{question}

\begin{question}
    Does a result analogous to Theorem \ref{refin} hold without the condition of boundedness of $(\omega(q_n))$?
\end{question}

It would also be interesting to obtain more examples of systems satisfying (\ref{rigid}), in particular ones related to the questions we mentioned at the end of Section \ref{intr}.
\begin{question}
    Are there uniquely ergodic systems satisfying (\ref{rigid}), for which
    \[
    \lim\limits_{n\to\infty}\sum\limits_{i<n}\frac{1}{n}f(T^{n^3}x)
    \]
    does not exist for some $x\in X$?
\end{question}
\begin{question}
    Are there mixing systems satisfying (\ref{rigid})?
\end{question}

\section{Notation and facts}
\subsection{General}

We use the convention $0\in\mathbb{N}$.

For two functions $f, g$ we write 
\[
f=O(g)\quad\iff\quad \limsup_{x\to\infty}\frac{\abs{f(x)}}{\abs{g(x)}}<\infty
\]
and
\[
f=o(g)\quad\iff\quad\lim_{x\to\infty}\frac{f(x)}{g(x)}=0.
\]
We also use $f\ll g$ to mean $f=O(g)$, and $f\approx g$ if $f\ll g\ll f$.

For $x\in\mathbb{R}$ we denote
\[
e(x)=e^{2\pi ix}
\]
and 
\[
||x||=\min_{a\in\mathbb{Z}}|x-a|.
\]

We will use the method of Poisson summation; see e.g. (\cite{iwaniec}, p.69 -- 71).

\begin{thm}
    For any Schwartz function $f\in \mathcal{S}(\mathbb{R})$ and any numbers $v\in\mathbb{R}^+$ and $u\in\mathbb{R}$ we have
    \[
    \sum\limits_{n\in\mathbb{Z}}f(vn+u)=\frac{1}{v}\sum\limits_{m\in\mathbb{Z}}\widehat{f}\left(\frac{m}{v}\right)e\left(\frac{um}{v}\right).
    \]
\end{thm}

We will also use the following slight variation of van der Corput's inequality; for completeness we provide the proof.
\begin{thm}[van der Corput's inequality in the periodic case]\label{vdC} 
Let $(a_n)$ be a sequence of complex numbers of magnitude at most 1, and assume it is periodic with period $N$. Then for any $H\in\mathbb{Z}^+$ we have
\[
\frac{1}{N}\abs{\sum\limits_{n<N}a_n}\leqslant \left(\frac{2}{H}\sum\limits_{h<H}\abs{\frac{1}{N}\sum\limits_{n<N}a_n\overline{a_{n+h}}}\right)^{1/2}
\]
\end{thm}
\begin{proof}
    Using the Cauchy-Schwartz inequality and periodicity of $(a_n)$, we get
    \begin{gather*}
    \abs{\sum\limits_{n<N}a_n}^2=\abs{\frac{1}{H}\sum\limits_{n<N}\sum\limits_{h<H}a_{n+h}}^2\leqslant \frac{N}{H^2}\sum\limits_{n<N}\abs{\sum\limits_{h<H}a_{n+h}}^2=\frac{N}{H^2}\sum\limits_{n<N}\sum\limits_{k, l<H}a_{n+k}\overline{a_{n+l}}= \\ =\frac{N}{H^2}\sum\limits_{n<N}\sum\limits_{|h|<H}\abs{H-h}a_{n}\overline{a_{n+h}}\leqslant \frac{2N}{H}\sum\limits_{h<H}\abs{\sum\limits_{n<N}a_{n}\overline{a_{n+h}}}.
    \end{gather*}

\end{proof}

    Let us also state without proof a classical fact concerning continuity of the translation operator in the $L^1$ norm.

\begin{fact}\label{conttrans}
    Consider $f\in L^1(\mathbb{T})$, and a sequence of real numbers $(\alpha_n)$ converging to 0. Then 
    \[
   \lim\limits_{n\to\infty} \int_\mathbb{T} \abs{f(x)-f(x-\alpha_n)}dx=0
    \]
\end{fact}

\subsection{Arithmetic functions}
We use several standard arithmetic functions throughout the text:
\begin{defn}
    For $n\in\mathbb{Z}^+$ we denote
    \begin{gather*}
    \varphi(n)=n\cdot\prod_{p|n}\left(1-\frac{1}{p}\right)\\
    \omega(n)=\sum\limits_{p|n}1\\
    \tau(n)=\sum\limits_{d|n}1\\
    \text{for $p$ prime} \quad \nu_p(n)=\max\{e\in\mathbb{N}: p^e|n\}
    \
    \end{gather*}
\end{defn}
We will use the following simple bound for the size of $\tau(n)$, proven e.g. in (\cite{hardy-wright}, Theorem 315, p. 343).
\begin{thm}\label{taubound}
    \[
    \tau(n)=o(n^{\delta})
    \]
    for any $\delta>0$.
\end{thm}

We will also need the notion of characters; for a broader overview see e.g. (\cite{iwaniec}, p.43 -- 45).

    A character of a finite abelian group $G$ is a homomorphism from $G$ to $\mathbb{C}^{*}$. The order of a character $\chi$ is the smallest $d>0$ for which $\chi^d=1$. Characters of $G$ form a group isomorphic to $G$, denoted $\widehat{G}$.

    We have the formula
    \begin{equation}\label{ortho}
    \sum\limits_{\chi\in \widehat{G}}\chi(g)=\begin{cases}
        |G|\quad\text{for $g=e_G$} \\ 0\quad\text{otherwise}.
    \end{cases}
    \end{equation}

    A Dirichlet character of modulus $m$ is a function $\chi:\mathbb{Z}\to\mathbb{C}$ defined by
    \[
    \chi(a)=\begin{cases}\chi'(a) \quad\text{when $\gcd(a, m)=1$}\\ 0\quad\text{otherwise}\end{cases}
    \]
   for a character $\chi'$  of the group $(\mathbb{Z}/m\mathbb{Z})^{*}$. We call $\chi$ primitive if it is of order 1.

    The relation (\ref{ortho}) implies
    \[
     \sum\limits_{\chi\in \widehat{G}}\chi(g)\overline{\chi(h)}=\begin{cases}
        |G|\quad\text{for $g=h$} \\ 0\quad\text{otherwise},
    \end{cases}
    \]
 for $h\in G$. In the language of Dirichlet characters this becomes
 \begin{equation}\label{ortho2}
  \frac{1}{\varphi(m)}\sum\limits_{\chi}\chi(x)\overline{\chi(t)}=\mathbf{1}_{m\mathbb{Z}+t}(x)
 \end{equation}
when $\gcd(m, t)=1$, where the sum is over all $\varphi(m)$ Dirichlet characters of modulus $m$.

We will use the following results about characters.

\begin{thm}\cite{burgess}\label{burg}
    For any natural $n, h$ and any nonprincipal character $\chi$ of modulus $n$ we have
    \[
    \sum\limits_{x<n}\left\lvert\sum_{i<h}\chi(x+i)\right\rvert^2<nh.
    \]
\end{thm}

\begin{thm}(\cite{keith}, Theorem A.1)\label{keith}
            Let $\chi$ be a nontrivial character of $\mathbb{F}_p^{*}$ order $k$, and let $\varepsilon_1, \varepsilon_2\in\mathbb{C}$ be any $k$-th roots of unity. Then for any $i\in \mathbb{F}_p^{*}$ the number
            \[
            N=|\{x\in\mathbb{F}_p: \chi(x)=\varepsilon_1, \chi(x+i)=\varepsilon_2\}|
            \]
            satisfies
            \[
            \abs{N-\frac{p}{k^2}}<\sqrt{p}+1.
            \]
        \end{thm}

\subsection{Dynamical systems}

Throughout the text we will focus on uniquely ergodic topological dynamical systems. By a topological dynamical system we mean a pair $(X, T)$, where $X$ is a compact metric space and $T:X\to X$ a homeomorphism. 

\begin{defn}
A topological dynamical system $(X, T)$ is uniquely ergodic if the Borel $\sigma$-algebra of $X$ admits a unique $T$-invariant measure.
\end{defn}

We use $\mu$ exclusively to denote this unique invariant measure. 

\begin{defn}
    A topological dynamical system $(X, T)$ is totally uniquely ergodic if the systems $(X, T^n)$ are uniquely ergodic for all $n\in\mathbb{Z}^+$.
\end{defn}

The measure $\mu$ provides a uniquely ergodic system $(X, T)$ with the structure of a measure-preserving dynamical system:
\begin{defn}
    A measure-preserving dynamical system is a collection $(X, \mathcal{B}, \mu, T)$, where $(X, \mathcal{B}, \mu)$ is a measure space with $\mu(X)=1$, and $T:X\to X$ is a measurable map which preserves the measure $\mu$.
\end{defn}

We will be interested in weakly mixing dynamical systems. One of the equivalent characterization of weak mixing (that we will use as definition in what follows) is

\begin{defn}
    A measure-preserving dynamical system $(X, \mathcal{B}, \mu, T)$ is weakly mixing if for any $A, B\in\mathcal{B}$ we have
    \[
    \lim\limits_{n\to\infty}\frac{1}{n}\sum\limits_{i<n}\abs{\mu(T^{-i}A\cap B)-\mu(A)\mu(B)}=0.
    \]
\end{defn}

\subsection{Special flows and skew product systems}
Here we define two classes of dynamical systems, to which we are able to apply our abstract results on sparse equidistribution. 
\begin{defn}[Special flows over rotations]
    For a continuous function $g:\mathbb{T}\to (0, \infty)$ we let $X_g$ be the quotient of $\mathbb{T}\times\mathbb{R}$ by the relation
    \[
    (x, y)\sim (x+\alpha, y-g(x)),
    \]
    and we equip it with the normalized Lebesgue measure and the metric induced from the Euclidean metric on $\mathbb{T}\times\mathbb{R}$. We denote by by $(T_t)$ the continuous flow on $X_g$ given by
    \[
    T_t([(x, y)])=[(x, y+t)],
    \]
    and we call it the special flow on $\mathbb{T}$ under $g$.
    
\end{defn}
For a more explicit equivalent definition, notice that we can identify $X_g$ with the set
\[
\mathbb{T}_g:=\left\{(x, y)\in \mathbb{T}\times \mathbb{R}: 0\leqslant y<g(x)\right\}
\]
(with the appropriate metric), since it contains precisely one representant of each equivalence class of $\sim$. In this case the flow moves each point up with unit speed, and identifies $(x, g(x))$ with $(x+\alpha, 0)$.
In particular, if we define the function $S_a(g):\mathbb{T}\to\mathbb{R}$ for $a\in\mathbb{Z}$ as
\[
S_a(g)(x)=\begin{cases}g(x)+\ldots + g(x+(a-1)\alpha)\quad\text{for $a>0$}\\ 0 \quad\text{for $a=0$}\\-g(x-\alpha)-\ldots -g(x-a\alpha)\quad\text{for $a<0$}\end{cases}
\]
then
\[
T_t(x, y)=\left(x+a\alpha, y+t-S_a(g)(x)\right)
\]
for the unique integer $a$ satisfying
\[
S_a(g)(x)\leqslant t+y<S_{a+1}(g)(x).
\]
The flow $(T_t)$ understood in this way preserves the Lebesgue measure on $X_g$, and this is the only measure preserved by this flow. Under a certain assumption on $g$, it turns out to also be the only measure preserved by the maps $T_C$:
\begin{fact}\label{meascond}
    Fix an irrational $\alpha$ and a continuous function $g:\mathbb{T}\to\mathbb{R}^+$, and assume that the equation
    \[
    rg(x)\equiv \psi(x+\alpha)-\psi(x)\pmod{1}
    \]
    has no measurable solutions for $r\in\mathbb{R}\setminus\{0\}$. Then the system $(X_g, T_C)$ is uniquely ergodic and weakly mixing for any $C>0$. 
\end{fact}
Indeed, under this assumption the flow $(T_t)$ is weakly mixing, i.e. for any measurable $A, B\subset X_g$ we have
\[
\frac{1}{T}\int_{0}^{T}\abs{\mu(A\cap T_{-t}(B))-\mu(A)\mu(B)}=0,
\]
as shown for example in (\cite{corinna}, Section 1.1.3). This implies that each map $T_C$ is weakly mixing (\cite{auer}, p.5) and uniquely ergodic (\cite{benda}, p.4).

\begin{defn}[Skew products over rotations]
    For an irrational number $\alpha$ and continuous function $g:\mathbb{T}\to\mathbb{T}$ we denote by $T_{\alpha, g}$ the transformation of $\mathbb{T}^2$ given by
    \[
    T(x, y)=(x+\alpha, y+g(x)).
    \]
\end{defn}
We again have a convenient way of checking unique ergodicity of the system.
\begin{fact}\label{meascond2}
    Fix an irrational $\alpha$ and a continuous function $g:\mathbb{T}\to\mathbb{R}^+$, and assume that the equation
    \begin{equation}\label{measeq}
    rg(x)\equiv \psi(x+\alpha)-\psi(x)\pmod{1}
    \end{equation}
    has no measurable solutions for $r\in\mathbb{Z}\setminus{0}$. Then the system $(\mathbb{T}^2, T_{\alpha, g})$ is uniquely ergodic.
\end{fact}
    Indeed, by (\cite{Anzai}, Theorem 2) this system is ergodic, and ergodic skew products over uniquely ergodic systems are also uniquely ergodic by a result of Furstenberg (see e.g. \cite{einsiedler}, Theorem 4.21).

\section{Strong rigidity}
In this section we wish to prove Theorem \ref{strrigid} and provide examples of systems satisfying its assumptions, in particular obtaining weakly mixing systems in which every square orbit is equidistributed.

\subsection{Main result}

We first establish a number-theoretic result concerning the distribution of residues of small squares in arithmetic progressions.
For this, we need a lemma about exponential sums of polynomials. 

\begin{lem}\label{polysum}
Fix an integer $k\geqslant 2$, and let $P$ be an integer polynomial of degree $k$ with leading coefficient $\ell$. Let $q$ be any positive integer. Then 
    \[
    \sum\limits_{x<q}e\left(\frac{P(x)}{q}\right)=O_k\left(q\left(\frac{\gcd(q, \ell)}{q^{1-\delta}}\right)^{1/2^{k-1}}\right)
    \]
     holds for all $\delta>0$, and if $k=2$ it holds also for $\delta=0$.
\end{lem}
\begin{proof}
    For a function $f:\mathbb{R}\to\mathbb{R}$ we define the discrete derivative $\Delta_h(f)$ by the formula
    \[
    \Delta_{h}(f)(x)=f(x+h)-f(x).
    \]
    We will show that averages
    \[
    A_n=\frac{1}{q^n}\sum\limits_{h_1, \ldots, h_{n}<q}\abs{\frac{1}{q}\sum\limits_{x<q}e\left(\frac{\Delta_{h_1} \ldots\Delta_{h_n}(P)(x)}{q}\right)}
    \]
    satisfy
    \begin{equation}\label{indst}
    \abs{\frac{1}{q}\sum\limits_{x<q}e\left(\frac{P(x)}{q}\right)}\leqslant 2\left(A_n\right)^{1/2^{n}}.
    \end{equation}
    for positive integers $n$.
    Indeed, by Theorem \ref{vdC} with $N=H=q$ we get 
    \[
A_n\leqslant \frac{1}{q^{n}}\sum\limits_{h_1, \ldots, h_{n}<q}\left(\frac{2}{q}\sum\limits_{h_{n+1}<q}\abs{\frac{1}{q}\sum\limits_{x<q}e\left(\frac{\Delta_{h_{n+1}}\Delta_{h_1} \ldots \Delta_{h_{n}}(P)(x)}{q}\right)}\right)^{1/2}.
\]
By the Cauchy-Schwarz inequality, and since $\Delta_{h_{n+1}}$ commutes with each $\Delta_{h_i}$, this is at most
\[
\left(\frac{1}{q^{n}}\sum\limits_{h_1, \ldots, h_{n}<q}\frac{2}{q}\sum\limits_{h_{n+1}<q}\abs{\frac{1}{q}\sum\limits_{x<q}e\left(\frac{\Delta_{h_1} \ldots\Delta{h_{n+1}}(P)(x)}{q}\right)}\right)^{1/2}=\sqrt{2A_{n+1}},
\]
so by induction we have
\[
A_1\leqslant 2(A_n)^{1/2^{n-1}}
\]
Theorem \ref{vdC} for $N=H=q$ gives
\[
\abs{\frac{1}{q}\sum\limits_{x<q}e\left(\frac{P(x)}{q}\right)}\leqslant \sqrt{2A_1},
\]
which by the above bound implies (\ref{indst}).
In particular, we have
\begin{equation}\label{indstk}
    \abs{\frac{1}{q}\sum\limits_{x<q}e\left(\frac{P(x)}{q}\right)}\leqslant 2\left(\frac{1}{q^{k-1}}\sum\limits_{h_1, \ldots, h_{k-1}<q}\abs{\frac{1}{q}\sum\limits_{x<q}e\left(\frac{\Delta_{h_1} \ldots\Delta_{h_k-1}(P)(x)}{q}\right)}\right)^{1/2^{k-1}}
    \end{equation}

Inductively we obtain
\[
\Delta_{h_1} \ldots\Delta_{h_{k-1}}(P)(x)=x\cdot \ell k!\prod_{i=1}^{k-1}h_i+a
\]
for some $a\in\mathbb{Z}$, so the inner sum in the right hand side of \ref{indstk} is $q$ when
\[
q|\ell k!\prod_{i=1}^{k-1}h_i,
\]
and 0 otherwise. It remains to estimate the number of tuples $(h_i)$ for which the divisibility holds. 

Let $q'=q/\gcd(q, k!\ell)$. For $k=2$ we get precisely $q/q'$ possible $h_1$, and the result follows. We now focus on the case $k>2$.

The divisibility holds precisely when there exist divisors $d_1, \ldots, d_{k-1}$ of $q'$ whose product is $q'$ and such that $d_i|h_i$ for all $i$.

For fixed $d_1, \ldots, d_{k-1}$ the number of such tuples $(h_i)$ is
\[
\prod_{i=1}^{k-1}\frac{q}{d_i}=\frac{q^{k-1}}{q'},
\]
so we can bound 
\[
\left(\frac{1}{q^{k-1}}\sum\limits_{h_1, \ldots, h_{k-1}<q}\abs{\frac{1}{q}\sum\limits_{n<q}e\left(\frac{\Delta_{h_1, \ldots, h_{k-1}}(P)(n)}{q}\right)}\right)^{1/2^{k-1}}\ll \left(\frac{\tau(q)^{k-1}}{q'}\right)^{1/2^{k-1}}\leqslant \left(\frac{k!\gcd(q, \ell)\tau(q)^{k-1}}{q}\right)^{1/2^{k-1}}
\]
which ends the proof, since $\tau(q)=o(q^{\varepsilon})$ for all $\varepsilon>0$ by Theorem \ref{taubound}.

\end{proof}

We now provide a lemma, which will be crucial to the proof of Theorem \ref{strrigid}. The proof of the lemma in the case $C=2$ was shown to us by Maksym Radziwiłł.
\begin{lem}\label{lem}
Fix a polynomial $P$ of degree $C\geqslant 2$ with integer coefficients, and a real number $\delta>0$. Then for numbers $x, a, t\in\mathbb{Z}$ and numbers $q, r, M, N\in\mathbb{Z}^+$ satisfying $\gcd(q, r)=1$ we have
\[
\sum\limits_{\substack{m\equiv a\pmod{r}\\m\in [x, x+M]}}\abs{\{n\in [t, t+N]: P(n)\equiv m\pmod{q}\}}=\left(1+o(1)\right)\frac{MN}{qr}
\]
as long as
\[
q^{1-(1-\delta)/(2^C+1)}=o\left(\frac{M}{r}\right)\quad\text{and}\quad q^{1-(1-\delta)/(2^C+1)}=o(N).
\]
 If $C=2$, the statement holds also with $\delta=0$.

\end{lem}

\begin{proof}
We wish to compute
$$
\sum_{\substack{m \equiv a \pmod{r} \\ m \in [x, x + M]}} \left ( \sum_{\substack{n\in [t, t+N]\\ P(n) \equiv m \pmod{q}}} 1 \right )
$$
Let us fix an $\varepsilon>0$, and introduce a smooth minorant $\Phi\leqslant 1$ supported in $[\varepsilon, 1-\varepsilon]$ with integral at least $1-\varepsilon$. The above sum is then bounded from below by a smoothed sum of the form
$$
\sum_{\substack{m \equiv a \pmod{r}}} \left ( \sum_{\substack{P(n) \equiv m \pmod{q}}} \Phi \left ( \frac{n-t}{N} \right ) \right ) \Phi \left ( \frac{m - x}{M} \right )
$$

We express the congruence condition using additive characters, thus getting
\begin{equation}\label{twosum}
\frac{1}{q} \sum_{\ell < q} \left ( \sum_{n} \Phi \left ( \frac{n-t}{N} \right ) e \left ( \frac{\ell P(n)}{q} \right ) \right ) \cdot \left ( \sum_{m \equiv a \pmod{r}} \Phi \left ( \frac{m - x}{M} \right ) e \left ( - \frac{\ell m}{q} \right ) \right ).
\end{equation}
We would now like to transform both sums using Poisson summation. For the second sum, we first write
\begin{equation}\label{aritspl}
 \sum_{m \equiv a \pmod{r}} \Phi \left ( \frac{m - x}{M} \right ) e \left ( - \frac{\ell m}{q} \right)=\sum_{\substack{y<qr \\ y \equiv a \pmod{r}}} e \left ( - \frac{\ell y}{q} \right)\sum\limits_{m\equiv y\pmod{qr}}\Phi \left ( \frac{m - x}{M} \right ).
\end{equation}
The inner sum is simply
\[
\sum\limits_{m}\Phi \left ( \frac{qrm+y - x}{M} \right ),
\]
which by Poisson summation equals
\[
\frac{M}{qr}\sum\limits_{v}\widehat{\Phi} \left ( \frac{v M}{q r} \right )e \left ( \frac{v (y-x)}{q r} \right ).
\]
Therefore the whole sum (\ref{aritspl}) equals 

\begin{equation}\label{poisson1}
  \frac{M}{q r} \sum_{v} \left ( \sum_{\substack{y<qr \\ y \equiv a \pmod{r}}} e \left ( - \frac{\ell y}{q} + \frac{v (y-x)}{q r} \right ) \right ) \widehat{\Phi} \left ( \frac{v M}{q r} \right ).
\end{equation}

Since $\gcd(r,q) = 1$ we can write $y \equiv a \overline{q} q + b r \pmod{q r}$ with $b \pmod{q}$, where $\overline{q}$ is any integer such that $\overline{q}q\equiv 1\pmod{r}$. Therefore the sum over $y$ is equal to
$$
\sum_{b <q} e \left ( - \frac{\ell b r}{q} + \frac{av \overline{q}}{r} + \frac{b v}{q} -\frac{vx}{qr} \right ) = e \left ( \frac{av \overline{q}}{r} -\frac{vx}{qr} \right ) q \cdot \mathbf{1}_{v \equiv \ell r\pmod{q}}.
$$
Let $\widetilde{\ell}$ be the unique number in $(-q/2, q/2]$ satisfying $\widetilde{\ell}\equiv \ell r\pmod{q}$. By the above formula, in the sum over $v$ in (\ref{poisson1}) the only $v\in (-q/2, q/2]$ that contributes a nonzero amount is $v=\widetilde{\ell}$. Since $\Phi$ is smooth, $\widehat\Phi(t)=O(|t|^{-4})$ holds, so we have 
\[
\sum\limits_{|v|\geqslant q/2}\abs{\widehat\Phi\left(\frac{vM}{qr}\right)}\ll \sum\limits_{|v|\geqslant q/2}\left(\frac{|v|M}{qr}\right)^{-4}\ll \left(\frac{M}{qr}\right)^{-4}\cdot q^{-3},
\]
which is $o\left(q^{-1}\right)$ since $\sqrt{q}=o(M/r)$. Hence the sum over $m$ in (\ref{twosum}) is
\[
\frac{M}{r} \widehat{\Phi} \left ( \frac{M \widetilde{\ell}}{q r} \right ) e \left ( - \frac{\widetilde{\ell}x}{q r} + \frac{a\widetilde{\ell}\cdot\overline{q}}{r} \right ) + o\left(\frac{M}{qr}\right).
\] 
Notice that
\[
\frac{1}{q} \sum_{\ell<q} \left ( \sum_{n} \Phi \left ( \frac{n-t}{N} \right ) e \left ( \frac{\ell P(n)}{q} \right ) \right )=O(N),
\]
so (\ref{twosum}) is
\[
o\left(\frac{MN}{qr}\right)+\frac{1}{q} \sum_{\ell<q} \left ( \sum_{n} \Phi \left ( \frac{n-t}{N} \right ) e \left ( \frac{\ell P(n)}{q} \right ) \right )\cdot\frac{M}{r} \widehat{\Phi} \left ( \frac{M \widetilde{\ell}}{q r} \right ) e \left ( - \frac{\widetilde{\ell}x}{q r} + \frac{a\widetilde{\ell}\cdot\overline{q}}{r} \right ).
\]
We now turn our attention to the sum over $n$, which we treat it similarly to the sum over $m$. We write
\begin{align*}
 \sum_{n} \Phi \left ( \frac{n}{N} \right ) e \left ( \frac{\ell P(n+t)}{q} \right ) = \sum_{x<q}e \left ( \frac{\ell P(x+t)}{q} \right )\sum_{n\equiv x \pmod{q}}\Phi \left ( \frac{n}{N} \right ) 
\end{align*}
and apply Poisson summation to the inner sum, obtaining 
\[
\sum_{n\equiv x \pmod{q}}\Phi \left ( \frac{n}{N} \right )=\sum_{n}\Phi \left ( \frac{nx+q}{N} \right )=\frac{N}{q} \sum_{v} \widehat{\Phi} \left ( \frac{N v}{q} \right )e \left (  \frac{v x}{q} \right )  .
\]
Hence the sum over $n$ in (\ref{twosum}) is 
\[
\frac{N}{q} \sum_{v} \left ( \sum_{x<q} e \left ( \frac{\ell P(x+t)}{q} + \frac{v x}{q} \right ) \right ) \widehat{\Phi} \left ( \frac{N v}{q} \right ).
\]

We write
$$
\mathcal{G}(\ell, v, t; q) := \sum_{x<q} e \left ( \frac{\ell P(x+t)}{q} + \frac{v x}{q} \right ).
$$
Plugging everything together we see that (\ref{twosum}) is
\begin{equation}\label{finsum}
\frac{M N}{q^{2} r} \sum_{\ell<q}\sum_{ v}\widehat{\Phi} \left ( \frac{M \widetilde{\ell}}{q r} \right ) e \left ( - \frac{\widetilde{\ell}x}{q r} + \frac{a\widetilde{\ell}\cdot\overline{q}}{r}\right)\mathcal{G}(\ell, v, t; q) \widehat{\Phi} \left ( \frac{N v}{q} \right ) + o\left(\frac{MN}{qr}\right).
\end{equation}
We now notice that the term $\ell = 0 = v$ contributes
$$
\frac{M N}{q r} \widehat{\Phi}(0)^{2},
$$
which is the expected main term. 
 Since $\mathcal{G}(0, v, t; q) = q \mathbf{1}_{q | v}$, the terms with $\ell = 0, v \neq 0$ contribute at most 
 \[
\frac{MN}{q^2r}\widehat{\Phi}(0)\sum_{i\ne 0}q\abs{\widehat{\Phi}(iN)}=O\left(\frac{MN}{qr}\cdot N^{-1}\right),
 \]
 which is a negligible amount.

To prove that (\ref{finsum}) is $o(MN/qr)$ we just need to show that
\[
\sum_{\ell=1}^{q}\sum\limits_{v}\widehat{\Phi} \left ( \frac{M \widetilde{\ell}}{q r} \right ) e \left ( - \frac{\widetilde{\ell}x}{q r} + \frac{a\widetilde{\ell}\cdot\overline{q}}{r}\right)\mathcal{G}(\ell, v, t; q) \widehat{\Phi} \left ( \frac{N v}{q} \right )=o(q).
\]
By Lemma \ref{polysum} we have
\[
\abs{\mathcal{G}(\ell, v, t; q)}\ll q\left(\frac{\gcd(q, \ell)}{q^{1-\varepsilon}}\right)^{1/2^{C-1}}\leqslant q\left(\frac{|\widetilde{\ell}|}{q^{1-\delta}}\right)^{1/2^{C-1}}
\] 
where the last inequality follows from $\gcd(q, \ell)=\gcd(q, \widetilde{\ell})$. The function taking $\ell$ to $\widetilde{\ell}$ is a bijection between $[0, q)$ and $(-q/2, q/2]$ taking 0 to 0, and hence it is enough to show that
\begin{equation}\label{after}
\sum_{\ell\neq 0}\sum\limits_{v} \abs{\widehat{\Phi}\left( \frac{M\ell}{q r}\right) \widehat{\Phi}\left ( \frac{N v}{q} \right )\abs{\ell}^{1/2^{C-1}}} =o\left(q^{(1-\delta)/2^{C-1}}\right).
\end{equation}
But
\[
\sum_{\ell\neq 0} \abs{\widehat{\Phi}\left( \frac{M\ell}{q r}\right) \abs{\ell}^{1/2^{C-1}}}\ll \left(\frac{qr}{M}\right)^{1+1/2^{C-1}}+\sum\limits_{|\ell|>qr/M}\left(\frac{M|\ell|}{qr}\right)^{-2}\abs{\ell}^{1/2^{C-1}}\ll \left(\frac{qr}{M}\right)^{1+1/2^{C-1}}
\]
and similarly
\[
\sum_{v}\abs{\widehat{\Phi}\left(\frac{Nv}{q}\right)}\ll\frac{q}{N}.
\]
Our assumptions on the growth rate of $M/r$ and $N$ in terms of $q$ give
\[
\left(\frac{qr}{M}\right)^{1+1/2^{C-1}}\cdot \frac{q}{N}=o\left(q^{(1-\delta)/2^{C-1}}\right),
\]
so the desired estimate (\ref{after}) holds.

From this reasoning we conclude that the sum in the statement of the Lemma is at least 
\[
(1-\varepsilon+o(1))\frac{MN}{qr}.
\]
Considering a majorant of $\mathbf{1}_{[0, 1]}$ instead of a minorant gives an analogous upper bound, and since $\varepsilon$ was arbitrary, we get the claimed estimate.

If $C=2$, the same reasoning applies for $\delta=0$.

\end{proof}

For numbers $m, q\in\mathbb{N}$ and an interval $I$ we denote
\[
\Sq(q, I, m)=\abs{\{i\in I: i^2\equiv m\pmod{q}\}}.
\]
Notice, that if $m\equiv m'\pmod{q}$ we have
\[
\Sq(q, I, m)=\Sq(q, I, m').
\]

We are now ready to prove the main result.
\begin{proof}[Proof of Theorem \ref{strrigid}]
 We first show that we can strengthen (\ref{rigid}) to
\begin{equation}\label{rigibet}
\max\limits_{t< q_{n+1}^{4/5}\cdot C^2_{n+1}}\sup_{x\in X}d(x, T^{tq_n}x)=o(1).
\end{equation}
for some nondecreasing sequence $(C_n)$ approaching infinity. Indeed, let us denote
\[
d^{(n)}_k=\sup_{x\in X}d(x, T^{kq_n}x),
\]
and let
\[
B_n=\left(\max_{t<q_{n}^{4/5}}d^{(n-1)}_t\right)^{-1/3},
\]
which approaches infinity by our assumption. For any $k, l\in\mathbb{N}$ we have
\[
d(x, T^{(k+l)q_n}x)\leqslant d(x, T^{kq_n}x)+d(T^{kq_n}x, T^{(k+l)q_n}x)\leqslant d^{(n)}_k+d^{(n)}_l,
\]
so $d^{(n)}_{k+l}\leqslant d^{(n)}_k+d^{(n)}_l$.
For large $n$ any $t<q_{n+1}^{4/5}\cdot B_{n+1}^2$ can be written as a sum of at most $2B_{n+1}^2$ numbers $t_i<q_{n+1}^{4/5}$, and then
\[
d^{(n)}_{t}\leqslant d^{(n)}_{t_1}+d^{(n)}_{t_2}+\ldots \leqslant 2B_{n+1}^2\cdot B_{n+1}^{-3}=o(1),
\]
so the sequence $(B_n)$ satisfies (\ref{rigibet}). The sequence
\[
C_n=\min_{k\geqslant n} B_k
\]
also satisfies (\ref{rigibet}) and approaches infinity, and it is nondecreasing.

By adding a constant to $f$ we can assume that 
\[
\int_{X}f d\mu=0.
\]
Throughout the proof $x\in X$ will be fixed, and we will denote $g(i)=f(T^{i}x)$. Since $f$ is uniformly continuous, (\ref{rigibet}) implies

\begin{equation}\label{rigid2}
\max\limits_{t< q_{n+1}^{4/5}\cdot C^2_{n+1}}\sup_{i\in\mathbb{Z}}\abs{g(i)-g(i+tq_n)}=o(1).
\end{equation}

We will show that
\begin{equation}\label{key}
\max\limits_{I\in\mathcal{I}_n}\frac{1}{|I|}\sum\limits_{i\in I} g(i^2)=o(1),
\end{equation}
where $\mathcal{I}_n$ consists of all intervals $I$ with integer endpoints satisfying
\[
I\subset[0, q_{n+1}^{4/5}C_{n+1}]\quad\text{and}\quad q_n^{4/5}C_n\leqslant |I|\leqslant 2q_n^{4/5}C_n.
\]
Notice, that every interval of the form $[0, N-1]$ for $q_n^{4/5}C_n\leqslant N\leqslant q_{n+1}^{4/5}C_{n+1}$ can be partitioned into intervals satisfying the above condition, and applying (\ref{key}) to each interval and averaging we obtain
\[
\max\left\{\frac{1}{N}\sum\limits_{i<N}g(i^2): \,\, q_n^{4/5}C_n\leqslant N\leqslant q_{n+1}^{4/5}C_{n+1}\right\}=o(1),
\]
which implies the statement of the theorem. Therefore we just need to prove (\ref{key}).

Let $I=[A, B]\in \mathcal{I}_n$, and let $r$ be the largest multiple of $q_n$ smaller than $A^2$, so that $r\geqslant A^2-q_n$. For the purposes of this proof we will denote by $\overline{k}$ the residue of $k$ modulo $q_n$ for integer $k$. For $i\in I$ we have
\begin{equation}\label{grwrt}
i^2-(r+\overline{i^2})\leqslant B^2-(A^2-q_n)=(B+A)(B-A)+q_n\leqslant  2q_{n+1}^{4/5}C_{n+1}\cdot |I|+q_n<q_nq_{n+1}^{4/5}C_{n+1}^2
\end{equation}
for large enough $n$.
Therefore, since $q_n| i^2-r-\overline{i^2}$, by (\ref{rigid2}) we have 
\[
\sup_{i\in I}\abs{g(i^2)-g(r+\overline{i^2})}=o(1),
\]
so 
\[
\sum\limits_{i\in I}g(i^2)=o(|I|)+\sum\limits_{y<q_n}\Sq(q_n, I, y)\cdot g(r+y).
\]
Using (\ref{rigid2}) for the index $n-1$ we get that this is
\[
o(|I|)+\left\lfloor q_n^{4/5}C_n\right\rfloor^{-1}\cdot \sum\limits_{y<q_n}\sum\limits_{t<q_n^{4/5}C_n}\Sq(q_n, I, y)\cdot g(r+y+tq_{n-1}).
\]
The change of variables $z=y+tq_{n-1}$ gives
\[
\sum\limits_{t<q_n^{4/5}C_n}\sum\limits_{y=0}^{q_n-1}\Sq(q_n, I, y)\cdot g(r+y+tq_{n-1})=\sum\limits_{t<q_n^{4/5}C_n}\sum\limits_{z=tq_{n-1}}^{q_n+tq_{n-1}-1}\Sq(q_n, I, z-tq_{n-1})\cdot g(r+z).
\]
This sum differs from
\begin{equation}\label{varch}
\sum\limits_{t<q_n^{4/5}C_n}\sum\limits_{z=tq_{n-1}}^{q_n+tq_{n-1}-1}\Sq(q_n, I, \overline{z}-tq_{n-1})\cdot g(r+\overline{z})=\sum\limits_{z<q_n}g(r+z)\sum\limits_{t<q_n^{4/5}C_n}\Sq(q_n, I, z-tq_{n-1})
\end{equation}
by
\begin{equation}\label{difm}
\sum\limits_{t<q_n^{4/5}C_n}\sum\limits_{z=tq_{n-1}}^{q_n+tq_{n-1}-1}\Sq(q_n, I, z-tq_{n-1})\left(g(r+z)-g(r+\overline{z})\right).
\end{equation}
But for these $z$ we have
\[
z<q_n+q_n^{4/5}C_nq_{n-1}<q_nq_{n+1}^{4/5}C_{n+1}^2
\]
for large $n$, and so by (\ref{rigid2}) we have
\[
\abs{g(r+z)-g(r+\overline{z})}=o(1).
\]
Therefore, and since
\[
\sum\limits_{z=tq_{n-1}}^{q_n+tq_{n-1}-1}\Sq(q_n, I, z-tq_{n-1})=|I|,
\]
the sum (\ref{difm}) is $o(|I|q_n^{4/5}C_n)$, and so it suffices to show that (\ref{varch}) is $o(|I|q_n^{4/5}C_n)$.

By Lemma \ref{lem} we have 
\[
\sum\limits_{\substack{m\equiv z\pmod{q_{n-1}}\\m\in [z-q_{n-1}q_n^{4/5}C_n+1, z]}}\Sq(q_n, I, m)=\left(1+o(1)\right)\frac{q_{n-1}q_n^{4/5}C_n\cdot |I|}{q_nq_{n-1}},
\]
since $\gcd(q_n, q_{n-1})=1$ by our assumptions, and $|I|\geqslant q_n^{4/5}C_n$, which grows faster than $q_n^{4/5}.$ Hence (\ref{varch}) is
\[
\left(\sum_{z<q_n}g(r+z)\cdot \frac{|I| q_n^{4/5}C_n}{q_n}\right)+\sum\limits_{z<q_n}\abs{g(r+z)}\cdot o\left(\frac{|I| q_n^{4/5}C_n}{q_n}\right).
\]
Both summands are $o\left(|I| q_{n}^{4/5}C_n\right)$: the first one by unique ergodicity of $(X, T)$, and the second by boundedness of $g$.

Summarizing all the estimates, we see that 
\[
\sum\limits_{i\in I}g(i^2)=o(|I|).
\]
All the asymptotic estimates we used were uniform in the choice of $I\in\mathcal{I}_n$, so this implies (\ref{key}) and ends the proof.
\end{proof}

\subsection{Examples}\label{examples}
 In this section we provide examples of systems satisfying the assumptions of Theorem \ref{strrigid} in the form of skew products and special flows over rotations.

We will now show a lemma providing conditions on the Fourier expansion of $g$ that imply unique ergodicity of the relevant systems, and give us control over the sums
\[
S_n(g)(x)=\sum\limits_{i=0}^{n-1}g(x+i\alpha)
\]
appearing naturally in both of our classes of examples. 

\begin{lem}\label{gcond}
Let $\alpha\in\mathbb{R}$ be an irrational number, and let $(q_n)$ be the sequence of denominators of convergents of $\alpha$. Let a sequence $(a_n)$ of real numbers satisfy 
\[
|a_n|=o\left(\frac{1}{q_nq_{n+1}^{4/5}}\right)
\] 
and assume that
\begin{equation}\label{acond}
\abs{a_n}\geqslant\frac{1}{{q_n^{4/5}}q_{n+1}}
\end{equation}
holds for infinitely many $n$. Then the function
\[
g(x)=\sum\limits_{k=1}^{\infty}a_k \cos(2\pi q_k x).
\]
is continuous, and we have
\begin{equation}\label{rig}
\sup_{x\in \mathbb{T}}\abs{S_{q_n}(g)(x)}=o\left(q_{n+1}^{-4/5}\right).
\end{equation}
Also, for $r\in\mathbb{R}\setminus\{0\}$ and $\beta\in\mathbb{R}$ the equation
\begin{equation}\label{meassol}
rg(x)=\psi(x+\alpha)-\psi(x)+\beta\pmod{1}
\end{equation}
has no measurable solutions $\psi: \mathbb{T}\to\mathbb{R}$.
\end{lem}

\begin{proof}
    We clearly have
    \[
    \sum\limits_{k=1}^{\infty}q_k|a_k|<\infty,
    \]
    so $g$ is continuous.

   We prove  (\ref{rig}) for the function 
    \[
    \tilde{g}(x)=\sum\limits_{k=1}^{\infty}a_ke(q_k x),
    \]
    whose real part is $g$, so that (\ref{rig}) for $g$ follows. Let
    \[
    \varepsilon_n=\sup\limits_{k\geqslant n}\abs{a_kq_kq_{k+1}^{4/5}},
    \]
    which converges to 0. We now bound the terms $k<n$, $k=n$ and $k> n$ separately.

    If $k< n$, we have
    \[
    S_{q_n}(a_ke(q_k x))=a_ke(q_k x)\frac{e(q_nq_k\alpha)-1}{e(q_k\alpha)-1}\ll |a_k|\frac{q_k\norm{q_n\alpha}}{\norm{q_k\alpha}}\ll |a_k|\frac{q_kq_{k+1}}{q_{n+1}}\leqslant\varepsilon_k\frac{q_{k+1}^{1/5}}{q_{n+1}}.
    \]
    One can easily see that $q_{k+2}\geqslant 2q_k$ for all $k$, so
\[
    \sum\limits_{k<n}S_{q_n}(a_ke(q_k x))\ll \frac{q_n^{1/5}}{q_{n+1}} \sum\limits_{k<n} \varepsilon_k\cdot 2^{-\frac{n-k}{10}}=o\left(\frac{q_n^{1/5}}{q_{n+1}}\right)
    \]
    
    For $k=n$ we have
    \[
    S_{q_n}(a_ne(q_nx))\ll |a_n|q_n=o(q_{n+1}^{-4/5}).
    \]
    Finally, for the remaining case $k>n$ we simply bound
    \[
    \sum\limits_{k>n}S_{q_n}(a_ke(q_k x))\leqslant q_n\sum\limits_{k>n}|a_k|\leqslant \varepsilon_n q_n\sum\limits_{k>n}\frac{1}{q_kq_{k+1}^{4/5}},
    \]
 which again by $q_{k+2}\geqslant 2q_k$ is as most
    \[
    \frac{\varepsilon_nq_n}{q_{n+1}q_{n+2}^{4/5}}\cdot 2\sum\limits_{k>n}2^{n+1-k}=o\left(\frac{q_n^{1/5}}{q_{n+1}}\right).
    \]
    
    Summing up the results in these three cases we see that the condition (\ref{rig}) holds.

     Now assume that the condition (\ref{meassol}) doesn't hold. Denote by $A\subset \mathbb{N}$ the set of indices satisfying (\ref{acond}), and consider $n\in A$. Let
    \[
    t_n=\left\lfloor \frac{1}{100\abs{ra_n}q_n}\right\rfloor\leqslant \frac{q_{n+1}}{100q_n^{1/5}}=o(q_{n+1}).
    \]
    Then we have $\norm{t_nq_n\alpha}=o(1)$, so by Fact \ref{conttrans}
    \[
    \lim_{\substack{n\to\infty\\n\in A}}\int_{\mathbb{T}}\abs{\psi(x+t_nq_n\alpha)-\psi(x)}dx=0.
    \]
    Iterating (\ref{meassol}) we get
    \[
     rS_{t_nq_n}(g)(x)\equiv \psi(x+t_nq_\alpha)-\psi(x)+t_nq_n\beta\pmod{1}\quad
    \]
    so
    \begin{equation}\label{int}
    \lim\limits_{\substack{n\to\infty\\n\in A}}\int_{\mathbb{T}}\norm{rS_{t_nq_n}(g)(x)-t_nq_n\beta}dx=0.
    \end{equation}
   
    We have already established that
    \[
    \sum\limits_{k\ne n}S_{q_n}(a_ke(q_k x))=o\left(\frac{q_n^{1/5}}{q_{n+1}}\right)
    \]
    uniformly in $x\in\mathbb{T}$, so since $t_n=O(q_{n+1}/q_n^{1/5})$ we have
    \[
   \sum\limits_{k\ne n}S_{t_nq_n}(a_ke(q_k x)) =\sum\limits_{k\ne n}\sum\limits_{i<t_n}S_{q_n}\left(a_ke\left(q_k\left(x+iq_n\alpha\right)\right)\right)=o(1),
    \]
    and so
    \[
    S_{t_nq_n}(g)(x)=S_{t_nq_n}\Big(a_n\cos(2\pi q_nx)\Big)+o(1)=a_n\frac{\sin(\pi t_nq_n^2\alpha)}{\sin(\pi q_n\alpha)}\cos\Big(2\pi q_nx+\pi\alpha q_n(t_nq_n-1)\Big)+o(1).
    \]
    If we denote 
   \[
   b_n=ra_n\frac{\sin(\pi t_nq_n^2\alpha)}{\sin(\pi q_n\alpha)},
   \]
   then we get
   \begin{equation}\label{equality}
   \int_{\mathbb{T}}\norm{rS_{t_nq_n}(g)(x)-t_nq_n\beta}dx=o(1)+\int_{\mathbb{T}}\norm{b_n\cos(2\pi q_nx)-t_nq_n\beta}dx.
   \end{equation}
   But 
    \[
   \abs{b_n}=\abs{ra_n\frac{\sin(\pi t_nq_n^2\alpha)}{\sin(\pi q_n\alpha)}}=\abs{ra_n\frac{\sin\left(\pi t_nq_n\norm{q_n\alpha}\right)}{\sin(\pi \norm{q_n\alpha})}}\in \left[\frac{1}{1000}, \frac{1}{10}\right],
    \]
    so we just need to formalize the intuition that a cosine wave is far away from any constant in the $L^1$ norm. We can for example use the inequality
    \begin{gather*}
    \norm{b_n\cos(2\pi q_nx)-t_nq_n\beta}+\norm{b_n\cos(2\pi q_nx+\pi)-t_nq_n\beta}\geqslant \\ \geqslant\norm{b_n\cos(2\pi q_nx)-b_n\cos(2\pi q_nx+\pi)}=2\abs{b_n \cos(2\pi q_n x)}
    \end{gather*}
    for large enough $n\in A$, and so
    \begin{gather*}
        \int_{\mathbb{T}}\norm{b_n\cos(2\pi q_n x)-t_nq_n\beta}dx=\frac{1}{2}\int_{\mathbb{T}}\Big(\norm{b_n\cos(2\pi q_n x)-t_nq_n\beta}+\norm{b_n\cos(2\pi q_n x+\pi)-t_nq_n\beta}\Big)dx\geqslant \\ \geqslant \int_{\mathbb{T}} \abs{b_n\cos(2\pi q_n x)}=\frac{2}{\pi}b_n\geqslant \frac{1}{500\pi},
    \end{gather*}
    which together with (\ref{equality}) contradicts (\ref{int}).
\end{proof}

\begin{rem}
    The statement of the lemma holds even after perturbing $g$ by a function with quickly decreasing Fourier coefficients. In particular, if $g$ is as in the lemma, and 
    \[
    h(x)=\sum\limits_{k=1}^{\infty}b_k\cos(2\pi k x),
    \]
    where $(b_k)$ satisfies
    \[
    \sum\limits_{k<q_n}\frac{k|b_k|}{\norm{k\alpha}}+q_nq_{n+1}\sum\limits_{k\geqslant q_n}|b_k|=o\left(q_n^{1/5}\right),
    \]
    then the statement of the lemma holds for $g+h$. Indeed, the same calculations as in the proof give
    \[
    \sup_{x\in\mathbb{T}}\left(S_{q_n}(h)(x)\right)=o\left(\frac{q_n^{1/5}}{q_{n+1}}\right),
    \]
    so (\ref{rig}) is satisfied for $g+h$, and
    \[
    \sup_{x\in\mathbb{T}}\abs{rS_{t_nq_n}(h)(x)}=o(1)
    \]
    implying
    \[
    \int_{\mathbb{T}}\norm{rS_{t_nq_n}(g+h)(x)-t_nq_n\beta}dx=o(1)+\int_{\mathbb{T}}\norm{rS_{t_nq_n}(g)(x)-t_nq_n\beta}dx\geqslant o(1)+\frac{1}{500\pi}
    \]
    so (\ref{meassol}) has no solutions for $g+h$.
\end{rem}

\begin{rem}
    Sequences $(a_n)$ satisfying the required conditions exist exactly for those $\alpha$, for which 
    \[
    \frac{1}{{q_{k_n}^{4/5}}q_{k_n+1}}=o\left(\frac{1}{q_{k_n}q_{k_n+1}^{4/5}}\right)
    \]
    holds for some sequence $(k_n)$, which is equivalent to
    \begin{equation}\label{fracexp}
    \liminf\limits_{n\to\infty}\frac{q_n}{q_{n+1}}=0.
    \end{equation}
  This fails only when the continued fraction expansion of $\alpha$ has bounded coefficients, i.e. when $\alpha$ is a badly approximable number. Thus the set of $\alpha$ for which such sequences $(a_n)$ exist has in particular full Lebesgue measure.
   \end{rem}

Now we are finally in a position to present the relevant examples:

\begin{thm}\label{skew}
    If $\alpha$ is an irrational number and $g$ is a function satisfying the statement of Lemma \ref{gcond}, then each square orbit of $T_{\alpha, g}$ in $\mathbb{T}^2$ is equidistributed.
\end{thm}

\begin{proof}
    We check the assumptions of Theorem \ref{strrigid} for the sequence $(q_n)$ of denominators of $\alpha$. Clearly each two consecutive terms of $(q_n)$ are coprime.

    We have
\[
T^{q_n}(x, y)=\left(x+q_n\alpha, y+S_{q_n}(g)(x)\right),
\]
and the distance from this point to $(x, y)$ is at most
\[
\norm{q_n\alpha}+\abs{S_{q_n}(g)(x)}=o(q_{n+1}^{-4/5})
\]
by Lemma \ref{gcond}. By the triangle inequality we get the desired rigidity condition. Finally, by Fact \ref{meascond2} we have unique ergodicity.
\end{proof}

\begin{thm}\label{spec}
     Let $\alpha$ be an irrational number and $g$ a
    function satisfying the statement of Lemma \ref{gcond}. Let \[C>-\min_{x\in\mathbb{T}}g(x),\] and consider the special flow $(T_t)$ on $\mathbb{T}$ under $g+C$. Then the dynamical system $(X_{g+C}, T_C)$ is weakly mixing and each square orbit is equidistributed.
\end{thm}
\begin{proof}
    Again, we check the assumptions of Theorem \ref{strrigid} for the sequence $(q_n)$ of denominators of $\alpha$. 

    We have
    \[
    T_{q_nC}([(x, y)])=[(x, y+q_nC)],
    \]
    which by Lemma \ref{gcond} lies at most $o(q_{n+1}^{-4/5})$ away from
    \[
    \left[(x, y+q_nC+S_{q_n}(g)(x))\right]=[(x+q_n\alpha, y)],
    \]
    which in turn is at distance at most $\norm{q_n\alpha}=o(q_{n+1}^{-1})$ from $[(x, y)]$.
Unique ergodicity and weak mixing are given by Fact \ref{meascond}.

\end{proof}

\subsection{A comment on higher degree polynomials and other sequences}\label{other}

Here we discuss the limitations of the presented method and its applications to sequences other than squares.

Let us first mention, that the rigidity constraint (\ref{strrigid}) is close to the strongest one we could require:
\begin{fact}\label{rrr}
    If a system $(X, T)$ satisfies 
\[
\sup_{t<q_{n+1}}\sup_{x\in X}d(x, T^{tq_n}x)=o(1).
\]
    for some sequence $(q_n)$ satisfying $\gcd(q_n, q_{n+1})=1$,  then $T$ is the identity map.
\end{fact}
\begin{proof}
    Let $0\leqslant a_n<q_{n+1}$ be the multiplicative inverse of $q_n$ modulo $q_{n+1}$, so that there exists $0\leqslant b_n\leqslant q_n$
    \[
    a_nq_n=b_nq_{n+1}+1.
    \]
    By the triangle inequality, for any $x\in X$ we have
    \[
    d(x, Tx)\leqslant d(x, T^{a_nq_n}x)+d(Tx, T^{b_nq_{n+1}}(Tx))=o(1),
    \]
    so $Tx=x$.
\end{proof}

In the proof of Theorem \ref{strrigid}, besides the input of Lemma \ref{lem}, the only property of the sequence $(n^2)$ we used is a bound on the growth rate, which we used in (\ref{grwrt}). 

For an arbitrary sequence $(a_n)$ of positive integers, as long as an analogue of this lemma holds together with an appropriate growth rate, we can therefore expect an analogue of Theorem \ref{strrigid}.

To be specific, let
\[
h_1(N)=\max_{i<N}\abs{a_i-a_{i+1}},
\]
and assume that the asymptotic
\[
\sum\limits_{\substack{m\equiv a\pmod{r}\\m\in [x, x+M]}}\abs{\{i\in [t, t+N]: a_i\equiv m\pmod{q}\}}=\left(1+o(1)\right)\frac{MN}{qr}
\]
holds for parameters as in Lemma \ref{lem} as long as $M/r$ grows faster than $h_2(q)$, and $N$ faster than $h_3(q)$ for some functions, $h_2, h_3$. Then we can produce an analogue of Theorem \ref{strrigid} for $(a_n)$ with the rigidity condition
\[
\sup_{t<h(n)}\,\,\sup_{x\in X}d(x, T^{tq_n}x)=o(1),
\]
where
\[
h(n)=\max\left(h_2(q_{n+1}), \frac{h_3(q_n)}{q_n}h_1(h_3(q_{n+1}))\right)
\]
By Fact \ref{rrr} this condition is nontrivial only if $h(n)=o(q_{n+1})$.
In turn, if
\[
\max (h_2(n), h_1(h_3(n)))=o(n),
\]
even if only along primes, methods from Section \ref{examples} give nontrivial examples of skew product systems and special flows in which each $(a_i)$-orbit is equidistributed.

As it turns out, this constraint seems strong enough to exclude sequences $(P(n))$ for polynomials $P$ of degree at least 4. In particular, even for the sequence $(n^{4})$, we have
\[
h_1(N)\approx N^{3},
\]
but the method presented in Lemma \ref{lem}, even after restricting to prime $q$, only lets us get the result for $h_3(q)=q^{1/2+\varepsilon}$. We then have
\[
h(n)\geqslant \frac{q_n^{1/2+\varepsilon}}{q_{n}}\cdot q_{n+1}^{3/2+3\varepsilon}\geqslant q_{n+1}.
\]

Because of this, we believe new methods would be necessary to obtain analogous results along higher-degree polynomials.

\section{Weak rigidity}\label{weakrig}
In this section we prove Theorem \ref{refin}, stating that a much weaker rigidity condition implies equidistribution along a subsequence of the $(n^C)-$orbit. Throughout this section the exponent $C\geqslant 2$ will be a fixed integer.

\subsection{Approximations by combinations of scaled characters}

\begin{defn}
For natural $N$ let 
\[
\Pow_N(x)=\left\lvert\{t\in[1, N]: t^C\equiv x\pmod{N}\}\right\rvert
\]
and for $d|N$ let
\[
\Pow_N(x, d)=\left\lvert\{t\in[1, N]: \gcd(t, N)=d\quad\text{and}\quad t^C\equiv x\pmod{N}\}\right\rvert.
\]
\end{defn}

We would like to establish, that this function can be approximated by linear combinations of the following "scaled characters", with bounded coefficients and control over the number of terms.

\begin{defn} For integers $d, n$ we denote by $A(n, d)$ the set of all functions $h:\mathbb{N}\to\mathbb{C}$ of the form
\[
h(x)=\begin{cases} \chi\left(\frac{x}{d}\right)\quad\text{for $d|x$}\\0\quad\text{otherwise}\end{cases}
\]
for some character $\chi$ of modulus $n$.
\end{defn}
In particular, $A(N, 1)$ is the set of all characters of modulus $N$, and $A(1, N)$ is a singleton containing only the characteristic function of the set of multiples of $N$. 

\begin{lem}\label{scaled}
   Let $f_i\in A(n_i, d_i)$ for $1\leqslant i\leqslant k$, and assume that $\gcd(n_id_i, n_jd_j)=1$ for any $i\ne j$. Then 
    \[
    a\cdot \prod_{i=1}^{k}f_i\in A\left(\prod_{i=1}^{k}n_i, \prod_{i=1}^{k}d_i\right)
    \]
    for some $a\in\mathbb{C}$ of norm 1.
\end{lem}
\begin{proof}
    The general case follows by induction from the case $k=2$. Let $\chi_1$ and $\chi_2$ be the related characters of modulus $n_1$ and $n_2$ respectively. Then 
    \[    f_1f_2(d_1d_2x)=\chi_1(d_2x)\chi_2(d_1x)=\chi_1(d_2)\chi_2(d_1)\cdot \chi_1\chi_2(x),
    \]
    and $f_1f_2(x)$ is nonzero only if $d_1$ and $d_2$ divide $x$, so if $d_1d_2$ divides $x$. Hence setting $a=\overline{\chi_1(d_2)\chi_2(d_1)}$ works.
\end{proof}

\begin{lem}\label{undiv}
For any prime $p$ and natural number $e\geqslant 1$ the function $\Pow_{p^e}(x, 1)$ can be written as sum of at most $2C$ functions from $A(p^e, 1)$.
\end{lem}
\begin{proof}
Let $G=\left(\mathbb{Z}/p^e\mathbb{Z}\right)^{*}$, and $H$ be the subgroup of $G$ formed by all $C$-th powers of elements of $G$. By (\ref{ortho}) we have, that
\[
\sum\limits_{\chi\in\widehat{G/H}}\chi(g)=\begin{cases}[G:H]\quad\text{for $g=e_{G/H}$}\\ 0\quad\text{otherwise}\end{cases}
\]
for $g\in G/H$. If we denote by $h$ the standard homomorphism from $G$ to $G/H$, then we get
\[
\sum\limits_{\chi\in\widehat{G/H}}\chi\circ h(g)=\begin{cases}[G:H]\quad\text{for $g\in H$}\\ 0\quad\text{otherwise}\end{cases}
\]
for any $g\in G$. But raising to the $C$-th power is a homomorphism from $G$ onto $H$, so the preimage of any element of $H$ under it is of size $[G:H]$. Thus, treating $G$ as a subset of the interval $[1, p^e]$ of $\mathbb{Z}$, by the above equality we have
\[
\sum\limits_{\chi\in\widehat{G/H}}\chi\circ h(g)=|\{h\in G: h^C=g\}|=\Pow_{p^e}(g, 1)
\]
for $g\in G$. 

Since $\Pow_{p^e}(x, 1)$ is periodic with period $p^e$, and has value 0 whenever $p|x$, the $[G:H]$ Dirichlet characters corresponding to $\chi\circ h$ for $\chi\in \widehat{G/H}$ add up to $\Pow_{p^e}(x, 1)$ on all of $\mathbb{Z}$.
Thus, it remains to show that $[G:H]$ is at most $2C$.\\

For $p>2$ the group $G$ is cyclic of order $(p-1)p^{e-1}$. Then we can easily compute $[G:H]=\gcd(|G|, C)$. If $p=2$, then $G\cong C_2\times C_{2^{e-2}}$, where $C_k$ denotes the cyclic group of order $k$. In this case we again compute
\[
[G:H]=\begin{cases}2\gcd(2^{e-2}, C) \quad\text{if $2\mid C$}\\1\quad\text{otherwise}\end{cases}.
\]
In all cases clearly $[G:H]\leqslant 2C$.

\end{proof}
\begin{lem}\label{smpr}
    For any prime $p$ and fixed $0\leqslant f\leqslant e$, the function $\Pow_{p^e}(i, p^f)$ can be written as a sum of at most $2C\cdot p^{Cf}$ functions, each in $A(p^i, p^{e-i})$ for some $i\in [0, e]$. 
    
\end{lem}

\begin{proof}
    We split the proof into two cases.
    \begin{enumerate}
        \item ($Cf\geqslant e$) In this case, $\Pow_{p^e}(i, p^f)$ can only have positive value when $p^e|i$, and the value there is $p^{e-f}(1-\frac{1}{p})$ if $f<e$ and $1$ if $e=f$. Hence our function is a product of a constant of value at most $p^{e-f}<2C\cdot p^{Cf}$ and the unique function in $A(1, p^e)$.
        \item ($Cf<e$) In this case, $\Pow_{p^e}(i, p^f)$ can have positive value only if $\nu_p(i)=Cf$, and the value at such $i$ is equal to
        \[
        p^f\Pow_{p^{e-Cf}}\left(\frac{i}{p^{Cf}}, 1\right). 
        \]
        By Lemma \ref{undiv} this is the sum of at most $2C\cdot p^f$ functions from $A(p^{e-Cf}, p^{Cf})$. 
    \end{enumerate}
\end{proof}
\begin{lem}\label{approx}
    For any natural numbers $d|N$ there exists a function $h:\mathbb{N}\to\mathbb{C}$ such that
    \begin{equation}\label{mmm}
    \sum\limits_{i<N}\left\lvert\Pow_N(i)-h(i)\right\rvert\leqslant N\cdot \sum_{\substack{p|N\\\nu_p(N)>\nu_p(d)}}p^{-\nu_p(d)}
    \end{equation}
    and
    \[
    h=\sum\limits_{s\in S}a_sf_s,
    \]
    for some:
    \begin{itemize}
    \item set $S$ of size at most $(2C)^{\omega(N)}\cdot d^{C+1}$, 
    \item complex numbers $a_s$  of absolute value 1 for $s\in S$,
    \item functions $f_s$ for $s\in S$, each of which is in $A(N/d_s, d_s)$ for some $d_s|d$.
    \end{itemize}
\end{lem}
\begin{proof}
    We define $h$ as 
    \[
    h(i)=\sum\limits_{d'|d}\Pow_N(i, d').
    \]
    Then for any $i$ we have 
    \[
    Pow_N(i)=\sum\limits_{d'|N}\Pow_N(i, d')\geqslant h(i).
    \]
    Also, for any $d'|d$ we have
    \[
    \sum\limits_{i<N}\Pow_N(i, d')=\sum\limits_{i<N}|\{a\in [1, N]: a^C\equiv i\pmod N\}|=|\{a\in [1, N]: \gcd(a, N)=d'\}|,
    \]
    so the sum of the values of $h$ on the interval $[1, N]$ is just the number of $a\in[1, N]$ satisfying $\gcd(a, N)|d$. For each $a$ that does not satisfy this, there exists a prime $p|N$ for which $\nu_p(N)<\nu_p(d)$ and $p^{\nu_p(d)+1}|a$, so the number of such $a$ is at most
    \[
N\cdot \sum_{\substack{p|N\\\nu_p(N)<\nu_p(d)}}p^{-\nu_p(d)} 
\]
    proving (\ref{mmm}). \\

   By the Chinese remainder theorem for any positive integers $d_1|n_1$ and $d_2|n_2$ with $\gcd(n_1, n_2)=1$ we have
    \[
    \Pow_{n_1n_2}(i, d_1d_2)=\Pow_{n_1}(i, d_1)\cdot \Pow_{n_2}(i, d_2).
    \]
    Iterating this result we obtain
    \[
    \Pow_N(i, d')=\prod_{p|N}\Pow_{p^{\nu_p(N)}}\left(i, p^{\nu_{p_t}(d')}\right).
    \]
    We now apply Lemma \ref{smpr} to each of the $\omega(N)$ factors, and expand the product into a sum of at most 
    \[
   \prod_{p|N}2Cp^{C\nu_{p}(d')}=(2C)^{\omega(N)}\cdot (d')^C\leqslant (2C)^{\omega(N)}\cdot d^C
    \]
    functions, each of the form 
    \[
   \prod_{p|N}r_p\quad\text{for some}\quad r_p\in \bigcup_{a=0}^{\nu_p(N)}A\left(p^{a}, p^{\nu_p(N)-a}\right).
    \]
    By Lemma (\ref{scaled}) each such product is of the form $af$ for $a\in\mathbb{C}$ of absolute value 1, and $f\in A(N/\overline{d}, \overline{d})$ for some $\overline{d}|d'$, which implies $\overline{d}|d$. Therefore each $\Pow_N(i, d')$ for $d'|d$ is a sum of at most $(2C)^{\omega(N)}\cdot d^C$ functions of the desired form, so $h$ is a sum of at most $(2C)^{\omega(N)}\cdot d^{C+1}$ of them.
    \end{proof}

\subsection{Averaging products of scaled characters and $g$}

The main goal of this section is proving Lemma \ref{scalchar}, which concerns a function $g$ which is in some sense almost periodic and has small averages along arithmetic sequences. From this we achieve bounds on the average values of the products of $g$ and scaled characters. Eventually we will apply this to the function $g(n)=f(T^n(x))$.

We first extend Theorem \ref{burg} from bounding character sums over intervals of fixed length to arithmetic progressions of fixed length.

\begin{lem}\label{arithm} 
	Let $m, n$ be natural numbers and let $\chi$ be a character of modulus $n$. Assume, that $\chi$ is not periodic with period $m$.
	Then for any natural number $L$ we have
	\[\sum\limits_{x<n}\left\lvert\sum\limits_{i<L}\chi(x+im)\right\rvert<n\sqrt{mL}.
 \]
\end{lem}

\begin{proof}
Let $d=\gcd(m, n)$, and let $a\in\mathbb{Z}$. Integers $a$ satisfying 
    \[
    a\cdot \frac{m}{d}\equiv 1\pmod{\frac{n}{d}}
    \]
    form an arithmetic progression with difference $n/d$ and all terms coprime with $n/d$, so by the Chinese remainder theorem some such $a$ satisfies $\gcd(a, n)=1$. Then
    \[
  \sum\limits_{i<L}\chi(x+im)=\sum\limits_{i<L}\chi(ax+aim)=\chi(a)\sum\limits_{i<L}\chi(x+id).
    \]
   Notice that for $t\in\mathbb{Z}$ we have
   \[
  \sum\limits_{x<n}\mathbf{1}_{d\mathbb{Z}+t}(x)\left\lvert\sum\limits_{i<L}\chi(x+id)\right\rvert= \frac{1}{d}\sum\limits_{x<n}\left\lvert\sum\limits_{i<Ld}\chi(x+i)\mathbf{1}_{d\mathbb{Z}+t}(x+i)\right\rvert,
   \]
   since each term of the right double sum appears exactly $d$ times in the left double sum. If $\gcd(t, d)>1$ this is simply 0, and otherwise (\ref{ortho2}) gives 
   \[
    \mathbf{1}_{d\mathbb{Z}+t}(x)=\frac{1}{\varphi(d)}\sum\limits_{\eta}\eta(x)\cdot \overline{\eta(t)},
    \]
 where the sum is over the $\varphi(d)$ characters $\eta$ of modulus $d$. In this case
\begin{gather*}
\sum\limits_{x<n}\left\lvert\sum\limits_{i<Ld}\chi(x+i)\mathbf{1}_{d\mathbb{Z}+t}(x+i)\right\rvert=\frac{1}{\varphi(d)}\sum\limits_{x<n}\left\lvert\sum\limits_{i<Ld}\sum\limits_{\eta}\chi(x+i)\eta(x+i)\overline{\eta(t)}\right\rvert\leqslant \\ \leqslant\frac{1}{\varphi(d)}\sum\limits_{\eta}\sum\limits_{x<n}\left\lvert\sum\limits_{i<Ld}\chi(x+i)\eta(x+i)\right\rvert\leqslant \sqrt{n}\cdot \left(\frac{1}{\varphi(d)}\sum\limits_{\eta}\sum\limits_{x<n}\left\lvert\sum\limits_{i<Ld}\chi(x+i)\eta(x+i)\right\rvert^2\right)^{1/2},
\end{gather*}
where the last inequality follows by the Cauchy-Schwarz inequality.
But $\chi$ is not periodic with period $d|m$, so $\chi\eta$ is not principal, and so by Theorem \ref{burg} 
\[
\frac{1}{\varphi(d)}\sum\limits_{\eta}\sum\limits_{x<n}\left\lvert\sum\limits_{i<Ld}\chi(x+i)\eta(x+i)\right\rvert^2\leqslant nLd.
\]
Summarizing, we have
\begin{gather*}
\sum\limits_{x<n}\left\lvert\sum\limits_{i<L}\chi(x+im)\right\rvert=\sum\limits_{x<n}\left\lvert\sum\limits_{i<L}\chi(x+id)\right\rvert=\sum\limits_{t<d}\frac{1}{d}\sum\limits_{x<n}\left\lvert\sum\limits_{i<Ld}\chi(x+i)\mathbf{1}_{d\mathbb{Z}+t}(x+i)\right\rvert\leqslant \\\leqslant\sum\limits_{t<d}\frac{\sqrt{n}}{{d}}\cdot\sqrt{nLd}=n\sqrt{Ld}\leqslant n\sqrt{Lm}.
\end{gather*}

\end{proof}
	Notice, that in the case when $\chi$ is periodic with period $m$, there is no cancellation in the sum whatsoever. 
 \begin{lem}\label{scalchar}
     Let $n, m, L, d$ be positive integers such that $d|m, n$, and let $f\in A(n/d, d)$. Let $r=\lfloor{n/m}\rfloor$. Assume that for some function $g:\mathbb{N}\to[-1, 1]$ we have
     \[
     \max_{t<L}\sup_{x\in\mathbb{N}}\left\lvert g\left(x+tm\right)-g(x)\right\rvert\leqslant\varepsilon\]
   and
   \[\frac{1}{r}\left\lvert\sum\limits_{x<r}g(xm+t)\right\rvert\leqslant \varepsilon\quad\text{for $t\in \mathbb{N}$}. 
     \]
  Then we have  
   \[
     \frac{1}{n}\left\lvert\sum\limits_{x<n}g(x)f(x)\right\rvert\leqslant \frac{1}{d}\left(\sqrt{\frac{m}{L}}+\frac{2mL}{n}+\varepsilon\right).
     \]
 \end{lem}

\begin{proof}
    Let $\chi$ be the character of modulus $n/d$ related to $f$. Let $m'=m/d$ and $n'=n/d$, and let $g'$ be the function defined as $g'(x)=g(dx)$ for $x\in\mathbb{N}$. Since $f\in A(n', d)$ we have
    \[ \sum\limits_{x<n}g(x)f(x)=\sum\limits_{x<n'}g'(x)\chi(x).
    \]
    First assume, that $\chi$ is not periodic with period $m'$. In this case we bound
    \[
\frac{1}{n'}\left\lvert\sum\limits_{x<n'}g'(x)\chi(x)\right\rvert\leqslant \frac{1}{n'L}\abs{\sum\limits_{x<n'}\sum\limits_{i<L}g'(x+im')\chi(x)}+\frac{1}{L}\sum\limits_{t<L}\sup_{x\in\mathbb{N}}\abs{g'(x+tm')-g'(x)}.
    \]
    But for those $t$ we have
    \[
    \sup_{x\in\mathbb{N}}\abs{g'(x+tm')-g'(x)}=\sup_{x\in\mathbb{N}}\abs{g(dx+tm)-g(dx)}\leqslant \varepsilon,
    \]
    so the second average is bounded by $\varepsilon$. In the first we use the change of variables $y=x+im'$ to write
    \[
    \sum\limits_{i<L}\sum\limits_{x<n'}g'(x+im')\chi(x)=\sum\limits_{i<L}\sum\limits_{x=im'}^{n'+im'-1}g'(x)\chi(x-im').
    \]
    This sum and
    \[
    \sum\limits_{i<L}\sum\limits_{x<n'}g'(x)\chi(x-im')
    \]
    share all but
    \[
    \sum\limits_{i<L}im'<L^2m'
    \]
    terms, each of absolute value at most 1, so they differ by at most $2L^2m'$. But by Lemma \ref{arithm}
    \[
   \abs{ \sum\limits_{i<L}\sum\limits_{x<n'}g'(x)\chi(x-im')}\leqslant \sum\limits_{x<n'}\abs{\sum\limits_{i<L}\chi(x-im')}\leqslant n'\sqrt{m'L}.
    \]
    Summarizing the estimates we have
    \[
    \frac{1}{n}\left\lvert\sum\limits_{x<n}g(x)f(x)\right\rvert=\frac{1}{d}\cdot\frac{1}{n'}\left\lvert\sum\limits_{x<n'}g'(x)\chi(x)\right\rvert\leqslant \frac{1}{d}\left(\varepsilon+\frac{2L^2m'}{n'L}+\frac{n'\sqrt{m'L}}{n'L}\right),
    \]
    so the result holds in this case.

  Now assume that $\chi$ is periodic with period $m'$. Then
    \[
    \left\lvert\sum\limits_{x<n'}g'(x)\chi(x)\right\rvert\leqslant m'+\sum\limits_{t<m'}\left\lvert\sum\limits_{x<r}g'(xm'+t)\right\rvert=m'+\sum\limits_{t<m'}\left\lvert\sum\limits_{x<r}g(xm+td)\right\rvert\leqslant m'+rm'\varepsilon,
    \]
    so
 \[
  \frac{1}{n}\left\lvert\sum\limits_{x<n}g(x)f(x)\right\rvert\leqslant \frac{1}{d}\left(\frac{m}{n}+\frac{rm\varepsilon}{n}\right)\leqslant \frac{1}{d}\left(\frac{2mL}{n}+\varepsilon\right).
 \]

\end{proof}
In the application to our theorem, we will consider only small values of $d$, so the factor of $\frac{1}{d}$ in the obtained inequality will end up being replaced by 1 for simplicity.
\subsection{Main Result}

Instead of Theorem \ref{refin} we will prove the following stronger version of it.

\begin{thmalt}
Let $(X, T)$ be a totally uniquely ergodic measure preserving system. Fix $k\in\mathbb{N}$ and assume that for an increasing sequence $(q_n)$ of natural numbers we have
    \begin{equation}\label{rignew}
    \max_{t<q_n}\sup\limits_{x\in X}d\left(x, T^{tq_n}(x)\right)=o\left(1\right).\end{equation}

    Then for any $x\in X$ and any function $f\in C(X)$ we have 
    \[
    \lim_{\substack{n\to\infty\\\omega(n)\leqslant k}}\frac{1}{n}\sum\limits_{i<n}\Pow_{n}(i)\cdot f\left(T^{i}(x)\right)=\int\limits_{X}f \mathrm{d}\mu
    \]
    uniformly in $x\in X$.
    \end{thmalt}

    We first show that it is indeed stronger.

\begin{lem}
    Theorem \ref{refin}' implies Theorem \ref{refin}.
\end{lem}
\begin{proof}
Let $(X, T), C, k, (q_n), f$ be as in Theorem \ref{refin}. We can apply Theorem $\ref{refin}'$, obtaining
\[
\lim_{n\to\infty}\frac{1}{q_n}\sum\limits_{i<q_n}\Pow_{q_n}(i)\cdot f\left(T^{i}(x)\right)=\lim_{\substack{n\to\infty\\\omega(n)\leqslant k}}\frac{1}{n}\sum\limits_{i<n}\Pow_{n}(i)\cdot f\left(T^{i}(x)\right)=\int\limits_{X}f \mathrm{d}\mu
\]
uniformly in $x\in X$. Now fix $x\in X$ and let $g(i)=f(T^{i}(x))$ for $i\in\mathbb{N}$. By (\ref{rigweak}) and uniform continuity of $f$ we have
\[
\sup_{i\in\mathbb{N}}\max_{t<q_n^{C-1}}\abs{g(i+tq_n)-g(i)}=o(1).
\]
If by $\overline{a}$ we denote the residue modulo $q_n$ of $a$, then by (\ref{rigweak}) we have
\[
\max_{i<q_n}\abs{g\left(i^C\right)-g\left(\overline{i^C}\right)}=o(1),
\]
since for such $i$ we have $i^C<q_n\cdot q_n^{C-1}$ and $q_n|i^C-\overline{i^C}$. Then
\[
\frac{1}{q_n}\sum\limits_{i<q_n}g\left(i^C\right)=o(1)+\frac{1}{q_n}\sum\limits_{i<q_n}g\left(\overline{i^C}\right)=o(1)+\frac{1}{q_n}\sum\limits_{i<q_n}\Pow_{q_n}(i)g(i),
\]
so
\[
\lim\limits_{n\to\infty}\frac{1}{q_n}\sum\limits_{i<q_n}g\left(i^C\right)=\int_{X}f d\mu.
\]
The estimates on $g$ we used were uniform in $x$, so the above convergence is as well.

\end{proof}

\begin{proof}[Proof of Theorem \ref{refin}']
    Subtracting a constant from $f$ we can assume that its integral over $X$ is 0, and by scaling we can assume that the image of $f$ lies in the unit disk. Also, the result for $Re(f)$ and $Im(f)$ implies the result for $f$, so we can assume that  $f:X\to[-1, 1]$.  
    
    Fix an $\varepsilon>0$, and let 
    \[
    M=\left\lceil\varepsilon^{-1}\right\rceil\left(\prod_{p<\varepsilon^{-1}}p\right)^{\left\lceil\varepsilon^{-1}\right\rceil}
    \]
    and $q_n'=Mq_n$ for all $n\in\mathbb{N}$. By the same reasoning as at the beginning of the proof of Theorem \ref{strrigid} we obtain from (\ref{rignew}) that
    \[
    \lim\limits_{n\to\infty}\max_{t<M^{2C+6}q_n}\,\,\sup\limits_{x\in X}d\left(x, T^{tq_n}(x)\right)\leqslant \lim\limits_{n\to\infty}M^{2C+6}\max_{t<q_n}\sup\limits_{x\in X}d\left(x, T^{tq_n}(x)\right)=0,
    \]
    so if we fix $x\in X$ and let $g(i)=f\left(T^{i}(x)\right)$, we have
    \[
    \lim\limits_{n\to\infty}\max_{t<M^{2C+4}q_n
'}\,\,\sup\limits_{r\in \mathbb{N}}\abs{g(r+tq_n')-g(x)}=0.
    \]
    Therefore for some $n$ we have
    \begin{equation}\label{epsone}
   \max_{t<M^{2C+4}q_n
'}\,\,\sup\limits_{r\in \mathbb{N}}\abs{g(r+tq_n')-g(x)}<\frac{\varepsilon}{(2C)^k\cdot M^{C+1}}.
    \end{equation}
    Now for any $N$ with $\omega(N)\leqslant k$ we can apply Lemma \ref{approx} to numbers $\gcd(N, M)|N$. Any prime $p$ for which $\nu_p(\gcd(M, N))<\nu_p(N)$ must satisfy $\nu_p(N)>\nu_p(M)$, so either $\nu_p(M)>\varepsilon^{-1}$ or $p>\varepsilon^{-1}$. Therefore the lemma gives
    \begin{equation}\label{lemapp}
    \frac{1}{N}\sum\limits_{i<N}\abs{\Pow_N(i)-h(i)}\leqslant \sum_{\substack{p|N\\\nu_p(N)>\nu_p(M)}}p^{-\nu_p(M)}\leqslant k\varepsilon+\frac{k}{2^{\varepsilon^{-1}}}<2k\varepsilon,
    \end{equation}
   where $h=\sum_{s\in S}a_sf_s$ as in Lemma \ref{approx}.
   By total unique ergodicity of $(X, T)$ we have
   \begin{equation}\label{epstwo}
   \max_{r\in\mathbb{N}}\frac{q_n'}{N}\left\lvert\sum\limits_{t<N/q_n'}g(tq_n'+r)\right\rvert\leqslant \frac{\varepsilon}{(2C)^k\cdot M^{C+1}}
    \end{equation}
    for large $N$, and then we can apply Lemma \ref{scalchar} to $n=N, m=q_n', L=M^{2C+4}q_n'$ and $d=d_s$, with $f=f_s$, obtaining
    \[
     \frac{1}{N}\left\lvert\sum\limits_{i<N}g(i)f_s(i)\right\rvert\leqslant \sqrt{\frac{1}{M^{2C+4}}}+\frac{2M^{2C+4}(q_n')^2}{N}+\frac{\varepsilon}{(2C)^k\cdot M^{C+1}},
    \]
    since (\ref{epsone}) and (\ref{epstwo}) are satisfied.
    For large $N$ this is bounded by 
    \[
    \sqrt{\frac{1}{M^{2C+4}}}+\frac{2\varepsilon}{(2C)^k\cdot M^{C+1}},
    \]
    and so since $|S|\leqslant (2C)^k\cdot M^{C+1}$ we get
    \[
     \frac{1}{N}\left\lvert\sum\limits_{i<N}g(i)h(i)\right\rvert\leqslant\sum\limits_{s\in S}  \frac{1}{N}\left\lvert\sum\limits_{i<N}g(i)a_sf_s(i)\right\rvert\leqslant \frac{2C^k}{M}+2\varepsilon\leqslant 4C^k\varepsilon,
    \]
    since $M>\varepsilon^{-1}$. Finally by (\ref{lemapp}) we get
    \[
     \frac{1}{N}\left\lvert\sum\limits_{i<N}g(i)\Pow_N(i)\right\rvert \leqslant \frac{1}{N}\left\lvert\sum\limits_{i<N}g(i)h(i)\right\rvert+2k\varepsilon<8kC^k\varepsilon
    \]
    for large $N$ satisfying $\omega(N)\leqslant k$. Since $\varepsilon$ was arbitrary, we get
    \[
      \lim\limits_{\substack{n\to\infty \\ \omega(n)\leqslant k}}\frac{1}{N}\left\lvert\sum\limits_{i<N}g(i)\Pow_N(i)\right\rvert=0.
    \]
    All the estimates we used were uniform in the choice of $x\in X$, so the above convergence is as well.
    
    \end{proof}

   \subsection{Counterexamples}
   
   In this section we wish to argue the strength of Theorem \ref{refin} by providing counterexamples to its potential strengthenings. We still treat $C\geqslant 2$ as a fixed integer.

   The counterexamples in this section are adapted from the one provided in \cite{PNT}. Because of this, instead of the full proof of the next lemma, we only highlight the necessary modifications to the proof of Lemma 11.6 in \cite{PNT}. 

   First of all, instead of a single set $A$ we consider a sequence of finite sets $A_n$:

   \begin{defn}
       We call a sequence $(A_n)$ of sets of integers \textit{sparse} if 
       \[
       \lim_{n\to\infty}\min_{\substack{x, y\in A_n\\x\ne y}}|x-y|=\infty.
       \]
       
   \end{defn}

   We also need the following
   \begin{defn}
       We say a set $B\subset\mathbb{N}$ is $C$-large if $B\cap (a\mathbb{N}+r)\cap(C\mathbb{N}+1)$ is nonempty for any positive integers $a, r$ whenever $\gcd(a, rC)=1$.
   \end{defn}

   \begin{lem}\label{counter}
       Let $(A_n)$ be a sparse sequence of sets of positive integers, $B_1, B_2$ be $C$-large, and $h:\mathbb{N}\to\mathbb{R}^+$. Then there exists a function $g\in C(\mathbb{T})$ and an irrational number $\alpha$ with denominators $(q_n)$, such that there is an increasing sequence $(t_n)$ of positive integers such that:
       \begin{itemize}
       \item $T_{\alpha, g}$ is totally uniquely ergodic
       
        \item for any $\varepsilon>0$ for large enough $n$ we have that for all $k\in A_{q_{t_{2n}}}\cap [q_{t_{2n}}/3^C, q_{t_{2n}}]$ 
        \[
           S_k(g)(0)\in (-\varepsilon, \varepsilon)\quad\text{if $n$ is even}
        \]
        and
        \[
           S_k(g)(0)\in \left(\frac{1}{2}-\varepsilon, \frac{1}{2}+\varepsilon\right)\quad\text{if $n$ is odd}
        \]
        \item $\sup\limits_{x\in\mathbb{T}}\norm{S_{q_{t_{2n+1}}}(g)(x)}=o(h(n)).$
        \item $q_{t_{2n}}\in B_1$ and $q_{t_{2n+1}}\in B_2$ for all $n$.
        \item each two terms of $(q_{t_n})$ are coprime.
           
       \end{itemize}
   \end{lem}
\begin{proof}
    We first take $\alpha$ such that $q_{n+1}\geqslant 2^nq_n$ and such that the sequence $(q_n)$ contains infinitely many elements of $B_1$ and $B_2$. This can be done by the following fact:
    \begin{fact}
        An nondecreasing sequence of positive integers $(q_n)$ is a sequence of denominators of some irrational $\alpha$ if and only if it satisfies the congruences
        \[
        q_{n+2}\equiv q_n\pmod{q_{n+1}}
        \]
        for $n\geqslant 0$, where we take $q_0=1$.
    \end{fact}
    Indeed, we can define $(q_n)$ recursively, making sure that each $q_n$ is congruent to 1 modulo $C$ and coprime with all previous $q_i$, and by $C$-largeness letting $q_{2n}\in B_1$ and $q_{2n+1}\in B_2$ for all $n$.

    We will now recursively define the sequence $t_n$ together with a sequence $l_n$ satisfying
    \[
    t_{2n}<l_n<t_{2n+1}
    \]
    for all $n$.

    We pick the terms $t_{2n}$ as in \cite{PNT}. There they only have to be large enough compared to the previously chosen terms, so we can ensure $q_{t_{2n}}\in B_1$ by the choice of $\alpha$.

    The sequence $(l_n)$ plays the same role as in subsection 11.2 of \cite{PNT}. To extend to total unique ergodicity, we set
    \[
    \widetilde{K_n}=K_n-(K_n \pmod{n}),
    \]
    to make it divisible by $n$. By our choice of $\alpha$ we have $K_n\geqslant 2^n$, so this change is insignificant enough to make $\widetilde{K_n}q_{l_n}\alpha$ still "close" to $\frac{1}{2q_{l_n}}$, so that in the end $S_{\widetilde{K_n}q_{l_n}}(g)$ is not close to any constant in measure. Since now any positive integer has infinitely many multiples in the sequence $(\widetilde{K_n})$, we get total unique ergodicity.

    To choose $q_{t_{2n+1}}$, we simply notice that by the calculations from Section 11.2 of \cite{PNT} we can make it so large in terms of the already constructed objects, that
    \[
     \sup\limits_{x\in \mathbb{T}}S_{q_{t_{2n+1}}}\left(\sum_{m\leqslant n} \left(f_m(x+\alpha)-f_m(x)+h_m(x+\alpha)-h_m(x)\right)\right)=o(h(n)),
    \]
    and that we can pick the next values of $(t_n)$ and $(l_n)$ so large that in the end
    \[
    \sup\limits_{x\in \mathbb{T}}S_{q_{t_{2n+1}}}(g)(x)=o(h(n)).
    \]
    Again, by our choice of $\alpha$ we can ensure that $q_{t_{2n+1}}\in B_2$.
    
\end{proof}

   We now show, that there is no hope of obtaining a variant of Theorem \ref{refin} proving equidistribution of the orbits (as opposed to equidistribution along a subsequence).

   \begin{thm}
       For any function $h:\mathbb{Z}^+\to\mathbb{R}^+$ there exists a totally uniquely ergodic dynamical system $(X, T)$ and a sequence $(q_n)$ of primes such that
       \[
       \sup_{x\in X}d(x, T^{q_n}(x))=o(h(q_n)),
       \]
       but the limit
       \[
       \lim\limits_{n\to\infty}\frac{1}{n}\sum\limits_{i<n}f\left(T^{i^C}(x)\right)
       \]
       doesn't exist for some function $f\in C(X)$ and point $x\in X$.
   \end{thm}
   \begin{proof}
       We use Lemma \ref{counter} to find the appropriate skew product system, letting $B_1=B_2$ be the set of prime numbers (which is $C$-large by Dirichlet's theorem) and letting $A_n$ be the set of all $C$-th powers larger than $n/3^C$. We clearly get the appropriate rigidity rate.
       
       Let us take $k_n=q_{t_{2n}}$, and let $f(x, y)=\tilde{f}(y)$ for any continuous function $\tilde{f}:\mathbb{T}\to[0, 1]$ satisfying $\tilde{f}(0)=0$ and $\tilde{f}(1/2)=1$. Then for even $n$ we have
       \[
       \frac{1}{k_n}\sum\limits_{i<k_n}f\left(T^{i^C}(0, 0)\right)\leqslant \frac{1}{3}+\frac{1}{k_n}\sum\limits_{k_n/3<i<k_n}\tilde{f}\left(S_{i^C}(g)(0)\right)=\frac{1}{3}+o(1),
       \]
       but similarly for odd $n$ the average is at least $2/3+o(1)$, so the sequence of averages cannot converge.
       
   \end{proof}
   \begin{rem}
       Notice, that in this system in particular every  orbit along $C$-th powers is equidistributed along a subsequence (and in particular dense), but not every such orbit is equidistributed.
   \end{rem}
    We now discuss the condition of boundedness of $\omega(q_n)$ in Theorem \ref{refin}. We are not able to prove that this condition is necessary for the theorem to hold; however, in Theorem \ref{counterstr} we will show that the condition is necessary in Theorem \ref{refin}', which constituted the main part of the proof of Theorem \ref{refin}. Because of this, we believe that the condition is indeed necessary also in Theorem \ref{refin}.

   We first need the following
    \begin{lem}\label{sparsing}
       Let $n$ be a squarefree positive integer, whose prime divisors are all congruent to 1 modulo $C$. Let $A$ be the set of residues modulo $n$ of $C$-th powers coprime with $n$. Then there exists a set $A'\subset A$ such that
        \begin{itemize}
            \item $|A'|\leqslant |A|\cdot\omega(n)\cdot (7/8)^{\omega(n)/3-4C^4}$
            \item any two elements of $A\setminus A'$ differ by at least $\omega(n)$.
        \end{itemize}
    \end{lem}
    \begin{proof}
    By the Chinese remainder theorem we have 
    \[
    |A|=\prod_{p|n}\frac{p-1}{C}.
    \] We take $A'$ as the union of sets $A\cap (A-i)$ for $1\leqslant i\leqslant \omega(n)$. This clearly satisfies the second of the desired conditions.

        We estimate the number of distinct residues modulo $p|n$ the set $A\cap (A-i)$ can take.
        If $p|i$ or $p<4C^4$, we use the trivial upper bound of $(p-1)/C$ residues.
        Else we use Theorem \ref{keith}.
        
        To use this fact let us fix a generator $g$ of the multiplicative group $\mathbb{F}_p^{*}$, and consider the character
        \[
        \chi\left(g^k\right)=e(k(p-1)/C).
        \]
        Applying this theorem we now get, that $A\cap (A-i)$ has at most
        \[
        \frac{p}{C^2}+\sqrt{p}+1<p\left(\frac{1}{C^2}+\frac{1}{2C^2}+\frac{1}{8C^4}\right)<\frac{7(p-1)}{8C}
        \]
        residues modulo $p$. 
        
        At least $\omega(n)/3$ primes dividing $n$ are larger than $\omega(n)>i$, so this case encompasses at least $\omega(n)/3-4C^4$ primes $p$. 
        
        Using the Chinese remainder theorem we can therefore bound
        \[
        |A\cap (A-i)|\leqslant (7/8)^{\omega(n)/3-4C^4}\prod_{p|n}\frac{p-1}{C},
        \]
        which gives the desired bound.
    \end{proof}

We now present the example which shows, that bounding $\omega(n)$ by a constant in Theorem \ref{refin}' is necessary, even if we assume an arbitrarily slow growth of the number of divisors to infinity, arbitrarily good rigidity rate and consider only squarefree numbers.
    
    \begin{thm}\label{counterstr}
        Take any functions $h_1, h_2:\mathbb{Z}^+\to\mathbb{R}^+$, the second one approaching infinity. Then there exists a totally uniquely ergodic dynamical system $(X, T)$ and a sequence $(q_n)$ of primes satisfying
       \[
       \sup_{x\in X}d(x, T^{q_n}(x))=o(h_1(q_n)),
       \]
       and a sequence $(r_n)$ of squarefree integers, each two coprime, satisfying $\omega(r_n)<h_2(r_n)$, such that
       \[
       \lim\limits_{n\to\infty}\frac{1}{r_n}\sum\limits_{i<r_n}\Pow_{r_n}(i)\cdot f\left(T^{i}(x)\right)
       \]
       doesn't exist for some function $f\in C(X)$ and point $x\in X$.
    \end{thm}
\begin{proof}
    We want to apply Lemma \ref{counter}, so we define the relevant objects. We let $B_2$ be the set of prime numbers.

    Let $(a_n, r_n)$ enumerate the set $\{(a, r)\in(\mathbb{Z}^+)^2: \gcd(a, rC)=1\}$. By Dirichlet's theorem we find $n-1$ primes in the progression $a_n\mathbb{Z}+1$ and one in $a_n\mathbb{Z}+r_n$, all distinct and greater than
    \[
    n^2+\min\{k\in\mathbb{Z}^+: h_2(k')>n \text{ for } k'\geqslant k\},
    \]
    and all congruent to $1$ modulo $C$, and we let $b_n$ be their product. Then $\omega(b_n)=n<h_2(b_n)$.

    We set $B_1=\{b_n: n\in\mathbb{N}\}$, which is $C$-large by construction. For $n\in B_1$ we let $A_n$ and $A'_n$ be as in Lemma \ref{sparsing}; for other $n$ we take $A_n=A'_n=\emptyset$. Then the sequence $(A_n\setminus A'_n)$ is sparse, so we can apply Lemma \ref{counter} to it.

    We then get the sequence $(q_{t_{2n+1}})$ which satisfies the rigidity condition. Let $r_n=q_{t_{2n}}$. Then, since $r_n\in B_1$, they are squarefree and satisfy $\omega(r_n)<h_2(r_n)$. 
    
   By the Chinese remainder theorem we have
   \[
   \Pow_{r_n}(a)=C^{\omega(r_n)}
   \]
   for each $a\in A_{r_n}$.
   We now use Lemma \ref{lem} for $M=r_n/3$ and $N=r_n$, obtaining  
   \[
   \sum\limits_{\substack{m\in [0, r_n/3]}}\abs{\{n\in [1, r_n]: n^C\equiv m\pmod{r_n}\}}=\left(1+o(1)\right)\frac{r_n\cdot r_n/3}{r_n\cdot 1},
   \]
   which implies
   
    \[
    |A_{r_n}\cap [0, r_n/3]|\leqslant \frac{1}{C^{\omega(r_n)}}\sum\limits_{m\leqslant r_n/3}\Pow_{r_n}(m)=\frac{r_n(1+o(1))}{3C^{\omega(r_n)}}=(1/3+o(1))|A_{r_n}|.
    \]
    Since $\omega(r_n)$ approaches infinity, by the choice of $(A_n)$ and $(A'_n)$ we have
    \[
  |A_{r_n}'|=o(|A_{r_n}|).
  \]
    By Bernoulli's inequality, we have 
    \[
   \frac{1}{r_n}\sum\limits_{m\leqslant r_n}\abs{\Pow_{r_n}(i)-\Pow_{r_n}(i, 1)}=1-\prod_{p|r_n}\left(1-\frac{1}{p}\right)\leqslant 1-\left(1-\frac{1}{\omega(r_n)^2}\right)^{\omega(r_n)}\leqslant \omega(r_n)\cdot \omega(r_n)^{-2}=o(1).
    \]

    We can finally let $f(x, y)=\tilde{f}(y)$ for any continuous function $\tilde{f}:\mathbb{T}\to[0, 1]$ satisfying $\tilde{f}(0)=0$ and $\tilde{f}(1/2)=1$. Then for even $n$ we have
    \begin{gather*}
    \frac{1}{r_n}\sum\limits_{i<r_n}\Pow_{r_n}(i)\cdot f\left(T^{i}(x)\right)=o(1)+\frac{1}{r_n}\sum\limits_{i<r_n}\Pow_{r_n}(i, 1)\cdot f\left(T^{i}(x)\right)=o(1)+\frac{1}{|A_{r_n}|}\sum\limits_{a\in A_{r_n}}f(T^{a}(x))\leqslant \\ \leqslant o(1)+\frac{1}{3}+\frac{1}{|A_{r_n}|}\sum\limits_{\substack{a\in A_{r_n}\setminus A'_{r_n}\\ a>r_n/3}}f(T^{a}(x))\leqslant o(1)+\frac{1}{3}.   
    \end{gather*}
    but similarly for odd $n$ the average is at least $2/3+o(1)$, so the sequence of averages cannot converge.
\end{proof}

\section{Acknowledgments}
I would like to thank Adam Kanigowski, whose continued guidance and support throughout this research were invaluable. I am also very grateful to Maksym Radziwiłł, for providing the proof of the case $C=2$ of Lemma \ref{lem} and helpful discussions surrounding it.

\end{document}